\documentclass[runningheads]{llncs}
\usepackage[T1]{fontenc}
\usepackage{graphicx}
\usepackage[numbers]{natbib}
\usepackage{amssymb}
\usepackage{amsmath}
\usepackage{amsfonts}
\usepackage{mathtools}
\usepackage{algpseudocode}
\usepackage{tikz}

\usepackage{booktabs} 
\usepackage[linesnumbered,ruled]{algorithm2e} 

\usepackage{tikz}
\usetikzlibrary{decorations.pathmorphing} 
\usetikzlibrary{arrows.meta} 
\usetikzlibrary{patterns}
\usepackage{subcaption} %
\usepackage{amsmath} 
\usepackage{mathtools}
\usepackage{bbm}
\usepackage[dvipsnames]{xcolor}

\newcommand{\redcomment}[1]{\textcolor{black}{\textrm{#1}}}
\newcommand{\ma}[1]{\textcolor{black}{\textrm{#1}}}

\newcommand{\mathitsc}[1]{\text{\normalfont\itshape\scshape #1}}

\newcommand{\purplecomment}[1]{\textcolor{black}{\textrm{#1}}}

\SetAlFnt{\small}
\SetAlCapFnt{\small}
\SetAlCapNameFnt{\small}
\SetAlCapHSkip{0pt}
\IncMargin{-\parindent}

\usepackage[a4paper, left=1in, right=1in, top=1in, bottom=1in]{geometry}
\usepackage{graphicx}
\usepackage{float}
\begin{document}
\title{Network Topology That Excludes Braess’s paradox and Maintains Monotonicity in Flows Over Time\thanks{This work was supported by the National Natural Science Foundation of China (Nos. 72192804, 12671378, 12331014)}}
%
%
\author{Xujin Chen\inst{1,}\inst{2},  \ Xiyuan Deng \inst{1,}\inst{2}, \  Changjun Wang\inst{1}}
\institute{SKLMS, Academy of Mathematics and Systems Science, Chinese Academy of Sciences, Beijing, China \and School of Mathematical Sciences, University of Chinese Academy of Sciences, Beijing, China\\ \email{\{xchen, dengxiyuan, wcj\}@amss.ac.cn} }

\authorrunning{X. Chen et al.}
%
\maketitle              

\begin{abstract}

In the game of \emph{flow over time},  infinitesimal flow particles aim to travel from a source to a sink in a network as quickly as possible. \ma{Under} the Vickrey bottleneck model, the congestion effects on network edges are captured through FIFO queues, which arise when the inflow into an edge exceeds its capacity. This work addresses two open conjectures \ma{about} 
the game: 1) the characterization of networks that are immune to Braess's paradox, and 2) the monotonicity relationship between the network  
inflow rate and the overall flow makespan. We show that a single class of network topologies, \ma{called} \emph{chains of bipolar pseudo-arborescences} (\emph{BPAs}), \ma{fully resolves the first conjecture and yields partial progress on the second.}

 \hspace{3mm} Chains of BPAs constitute precisely \ma{the cases left open by Macko et al.~(2013) in their study of} Braess’s paradox for flow over time. By structurally characterizing the dynamic evolution of Nash flows over time on such networks, we prove that \redcomment{removing edges from these networks \ma{never decreases} the maximum equilibrium latency.} Combined with the results of \ma{Macko et al., this establishes} a necessary and sufficient condition: \emph{a network does not admit Braess’s paradox for flow over time if and only if it is a chain of BPAs}, thereby resolving their conjecture. Furthermore, we confirm the monotonicity conjecture \ma{of Correa et al.~(2021)} for all chains of BPAs. This result strictly generalizes the previously known monotonicity for chains of parallel paths \ma{under} uniform inflows. 



\end{abstract}





\section{Introduction}

Modern transportation and communication networks operate under constantly changing conditions, where congestion and travel latency fluctuate over time. In such environments, the decisions of individual agents influence not only their own travel but also the overall performance of the flow. Traditional static models fail to capture these temporal effects, ignoring queue formation and delays that develop as traffic moves through the network. Dynamic flow (or flows over time) addresses this limitation by modeling the movement of individual flow particles across the network over time, providing a framework to study the efficiency of decentralized routing and the resulting equilibrium behavior under time-dependent congestion.

Dynamic flow models were first studied by \citet{Ford1958,Ford1962} from an optimization perspective. Their work focused on computing a flow over time that maximizes the total amount sent from a source to a sink within a given time horizon. Since then, the literature on dynamic flows has expanded substantially. We refer to the survey by \citet{Skutella2009} for a comprehensive overview of subsequent developments. The study of dynamic flow games dates back to the seminal work of \citet{Vickrey1969}, which introduced \ma{what is now known as the \emph{Vickrey bottleneck model}}. In this model, each link is characterized by a capacity and a fixed transit time. As long as the inflow rate into a link does not exceed its capacity, the traversal time equals the transit time. When the inflow rate exceeds capacity over some period, however, a queue forms at the entrance of the link, and users experience a delay equal to the transit time plus the waiting time in the queue.  \ma{Since queues are induced by the collective routing decisions of all users, each user's travel time depends on the choices of others, which leads to a game in which users selfishly choose routes to minimize their own travel times. The Vickrey model has since become one of the most influential models of traffic congestion and has received considerable attention in the transportation science literature \citep{Yagar1971,Peeta2001}. However, a rigorous mathematical understanding of its equilibria has emerged only relatively recently.} 

\ma{The equilibrium concept in this setting is that of a} \emph{Nash flows over time} or \emph{equilibrium flows}. Over the past two decades, considerable effort has been devoted to understanding the structural and computational properties of Nash flows \citep{Koch2011,Macko2013,Cominetti2015,Sering2021,Cominetti2022}. \citet{Koch2011} provided an elegant characterization of the time derivatives of Nash flows under uniform inflow rates. Building on this characterization, \citet{Cominetti2015} \redcomment{established existence constructively for piecewise-constant inflow rates, together with uniqueness of earliest-arrival labels among right-continuous equilibria.} More recently, \citet{Cominetti2022} showed that, under a natural and necessary uniform inflow condition, Nash flows reach a steady state in finite time. \ma{Equilibria have also been studied in a variety of extensions and variants of the basic model} 
\citep{Sering2019nash,Scarsini2018,Graf2020,Cao2021,Graf2023}. Despite extensive study, many fundamental questions remain open, and a number of seemingly obvious properties have turned out to be surprisingly hard to establish.

\ma{A prominent example is \emph{Braess's paradox} \citep{Braess1968,Braess2005}, the counterintuitive phenomenon whereby adding resources to a network, such as new links or additional capacity, can worsen system performance at equilibrium. In its classic form, the paradox arises in static selfish routing, where adding a new link may lead users to reroute in a way that increases the latency experienced by every user \citep{Kameda2002,Lin2004,Roughgarden2006}. The paradox carries over to flows over time, where it exhibits markedly different behavior.}
\citet{Macko2013} showed that there exist networks that do not admit Braess’s paradox in static flow games but do admit it under the dynamic flow model, and further \ma{showed that, unlike in the static case, Braess’s paradox is not symmetric for flows over time---a network may admit the paradox while its reverse does not}. They also conjectured a necessary and sufficient condition \ma{on the network topology for the occurrence} of Braess’s paradox. To the best of our knowledge, the validity of this conjecture remains open.

\ma{Braess's paradox concerns how equilibria respond to changes in the network itself. An equally natural question is how they respond to changes in the inflow. Intuitively, sending the same amount of flow into the network at a higher rate should never delay its arrival at the sink. In light of Braess's paradox, however, such intuition cannot be taken for granted in dynamic flow games. This question lies at the heart of the \emph{monotonicity conjecture} proposed by \citet{Correa2021}, which states that, for any given instance, the time needed for a given amount of flow to reach the sink, i.e., the \emph{makespan}, in a Nash flow over time is non-increasing in the uniform inflow rate. This conjecture arose in the study of the \emph{price of anarchy} (PoA) for flows over time. The PoA is a standard measure of the inefficiency caused by selfish behavior, defined as the worst-case ratio between the performance (e.g., makespan or throughput) of a Nash equilibrium and that of an optimal solution \citep{koch2012routing,BHASKAR2015,Correa2021}. In the makespan setting, \citet{Correa2021}  showed that if the monotonicity conjecture holds, the PoA is exactly $e/(e-1)$, and they verified the conjecture for chains of parallel paths. Despite its intuitive appeal and supporting computational evidence, the monotonicity conjecture remains open in general.}

\paragraph{Our contribution.} In this work, we address the above two open conjectures \ma{\cite{Macko2013,Correa2021} by analyzing how Nash flows over time evolve on a specific class of network topologies}. A single-source single-sink graph is called a \emph{bipolar pseudo-arborescence} (\emph{BPA}) if every vertex other than the source and the sink has out-degree exactly one. A graph is called a \emph{chain of BPAs} if it is \ma{a BPA or obtained by sequentially connecting two or more BPAs} in series. We show that chains of BPAs are exactly \ma{the network topologies conjectured by \citet{Macko2013} to be necessary and sufficient for immunity to Braess's paradox. This topological class strictly generalizes chains of parallel paths, and serves as the fundamental structure needed to establish the monotonicity conjecture.}  

To prove that chains of BPAs do not admit Braess’s paradox, we first establish a series of structural results for a broader class of networks, namely series–parallel networks. \ma{Specifically, we} construct an iterative \emph{shortest-path flow decomposition} of such networks, which yields several structural characterizations and properties. \ma{\citet{BEIN1985} previously developed a greedy augmentation scheme that computes minimum-cost flows in series--parallel networks using only forward augmenting paths. Our decomposition, which also employs a forward-only augmentation scheme, serves a different purpose---rather than computing an optimal flow, it yields the structural characterizations needed for our subsequent analysis.}
We then study \ma{how a Nash flow over time evolves} on a single BPA, and show that the dynamically evolving shortest-path network coincides, at critical time points, with the subnetworks identified by our decomposition. These characterizations allow us to fully describe the flow dynamics and the resulting arrival flow rates on a single BPA. Building on this analysis, we derive the maximum travel latency of the Nash flow on chains of BPAs and show that Braess’s paradox cannot occur in these networks, thereby resolving the conjecture by \citet{Macko2013}. 

\ma{Building} on the full characterization of the flow dynamics and the resulting arrival flows on a single BPA, \ma{we establish a {dominance-preservation} property that}, for a fixed BPA \ma{network}, if one inflow dominates another pointwise in rate, then the corresponding arrival flow also dominates the other arrival flow pointwise in rate. We then extend this \ma{property} to chains of BPAs. 
\redcomment{This confirms the original constant-inflow monotonicity conjecture \ma{of \citet{Correa2021} for chains of BPAs, generalizing their result for chains of parallel paths. Moreover, our comparison is not limited to constant inflow rates; it extends} to nondecreasing, piecewise-constant inflows on their injection intervals. Here dominance is required only until the dominating inflow ends, as defined in Section~\ref{sec:mono}; it is not pointwise dominance on the entire time axis for two equal-volume inflows.}

\ma{In summary, we settle the topological conjecture of \citet{Macko2013} in full and make substantial partial progress on the monotonicity conjecture of \citet{Correa2021}.}


\paragraph{Paper organization.} 
The remainder of the paper is structured as follows. Section \ref{sec:pre} introduces the necessary preliminaries, including the flow over time model, Nash flows, Braess’s paradox, and the corresponding conjecture. \ma{Section~\ref{sec:topo} formally defines chains of BPAs and constructs the shortest-path flow decomposition of series-parallel networks, from which we derive the structural characterization needed later.} 
Section \ref{sec:nash} analyzes the dynamics of Nash flows on chains of BPAs and resolves the conjecture regarding Braess’s paradox. Section \ref{sec:mono} establishes the monotonicity conjecture for chains of BPAs. Section~\ref{sec:con} concludes the paper. 


\section{Preliminaries}\label{sec:pre}
All graphs studied in this paper are \emph{directed} unless otherwise noted.
\subsection{Flow Model}\label{sub:flow}
The network is modeled by a graph $G=(V,E)$ with vertex set $V=V(G)$ and edge set $E=E(G)$, where each edge $e\in E$ is associated with a positive capacity ${c}_e$ and a nonnegative length (a.k.a.\ transit time) $\tau_e$. \purplecomment{For any edge $e\in E$, we often use $\nu(e)$ and $\omega(e)$ to denote its tail and head vertices, respectively.} Let  $s,t\in V$ be the source and sink of $G$, respectively, such that every vertex in $G$ is reachable from $s$, and $t$ is reachable from every vertex of $G$. We often call $G$ an $s$-$t$ graph or a single-source single-sink graph. The flow enters $G$ at the source $s$ at a constant rate $\mu$\purplecomment{, starting at time $0$,} and travels from $s$ to the sink~$t$. We use $(G,s,t,c, \tau,\mu)$ to denote an instance of the dynamic \ma{flow game}. 
Throughout the paper, $G$ is assumed to be \emph{irredundant}, meaning that each edge (and hence each vertex) lies on at least an $s$-$t$ path. \ma{For each vertex  $v \in V$, let $\delta^-_G(v)$ and $\delta^+_G(v)$ denote the sets of edges entering and leaving $v$, respectively. If $G$ is clear from the context, we write $\delta^-(v)$ and $\delta^+(v)$ instead of $\delta^-_G(v)$ and $\delta^+_G(v)$.}

\begin{figure}[h]
        \centering
        \begin{tikzpicture}[scale=1]
    \draw[pattern=north east lines] (1,0) rectangle (4,0.8);
    
    \node at (2.5,0.4) {$q_e$};

    \draw (0,0) -- (1,0);
    \draw (0,0.8) -- (1,0.8);
    \draw[->][blue, ultra thick]  (4,0.4) -- (7,0.4) node[midway, below] {$e$}; 
    \draw[decorate, decoration={snake, amplitude=2pt, segment length=8pt}] (-0.1,0.4) -- (0.6,0.4);
    \draw[->] (0.6,0.4) -- (1,0.4);
    \node[left, inner sep=1pt] at (-0.3,0.3) {$f_e^+(\theta)$};

    \draw (7,0.4) -- (7.3,0.4);
    \draw[decorate, decoration={snake, amplitude=1.5pt, segment length=6pt}, ->] 
        (7.3,0.4) -- (7.8,0.4);
    \node[right, inner sep=1pt] at (8,0.3) {$f_e^-(\theta)$};
    
\end{tikzpicture}
        \caption{flow into an edge}
        \label{fig:edge}
    \end{figure}
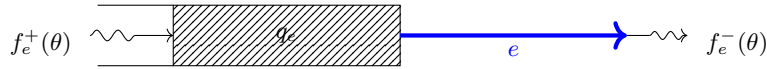

\ma{Initially (before time 0), the network has neither queues nor flow.} The flow over time, \ma{occurring since time 0}, can be described by a family of nonnegative and locally integrable functions \purplecomment{$f:= \{f_e^+, f_e^-\mid e \in E\}$}, where, given time (point) $\theta\ge0$, $f_e^+(\theta)$ and $f_e^-(\theta)$ denote the inflow rate and outflow rate into edge  $e$ at time $\theta$, respectively. 
\ma{The flow must satisfy the conservation constraint that, for almost every $\theta\ge0$, there holds}
\begin{equation*}
    \sum_{e \in \delta^+(v)} f_e^+(\theta) - \sum_{e \in \delta^-(v)} f_e^-(\theta) =
    \begin{cases}
        \mu & \text{if } v = s, \\
        0 & \text{if } v \ne s, t.
    \end{cases}
\end{equation*}

\ma{For any edge $e$ and time $\theta$, the total amount of flow that has entered and exited edge $e$ by time $\theta$ are the
cumulative inflow $F_e^+(\theta) $ and cumulative outflow $F_e^-(\theta) $, given by}
\[
F_e^+(\theta) := \int_0^\theta f_e^+(\vartheta) \, \mathrm{d}\vartheta \; \text{ and } \; F_e^-(\theta) := \int_0^\theta f_e^-(\vartheta) \, \mathrm{d}\vartheta,
\]
 respectively. Flow dynamics are modeled using the deterministic fluid queuing framework. In case the inflow rate $f_e^+(\theta)$ exceeds the edge capacity ${c}_e$, a queue grows at the tail of the edge at rate $f_e^+(\theta)-{c}_e$. The queue mass at time $\theta$, denoted by $z_e(\theta)$, \purplecomment{is the amount of flow that has entered edge $e$ but has not yet left its queue, i.e.,
\[
z_e(\theta):=F_e^+(\theta)-F_e^-(\theta+\tau_e).
\]}
If $f_e^+(\theta)<{c}_e$, \ma{$z_e(\theta)$} decreases at a rate ${c}_e-f_e^+(\theta)$ until 
\ma{it} becomes 0 or the inflow rate changes. \ma{Specifically, $z_e(\theta)$ follows the following deterministic queueing rule:
}
\[
\frac{\mathrm{d}z_e(\theta)}{\mathrm{d}\theta}=\begin{cases}
f_e^+(\theta)-{c}_e & \text{if } z_e(\theta) > 0, \\
\max \{f_e^+(\theta)-{c}_e,0\} & \text{if } z_e(\theta) = 0.
\end{cases}
\]
At time $\theta$, the flow
exits the queue and leaves $e$'s tail at rate ${c}_e$ if $z_e(\theta)>0$ and the flow leaves $e$'s tail at rate $\min\{f^+_e(\theta),{c}_e\}$ otherwise. As soon as the flow particle leaves $e$'s tail, it traverses $e$, from its tail to its head, in $\tau_e$ units of time. \ma{More precisely, the outflow satifies the rule:}
\[
f_e^-(\theta+\tau_e)=\begin{cases}
{c}_e & \text{if } z_e(\theta) > 0, \\
\min \{f_e^+(\theta), {c}_e\} & \text{if } z_e(\theta) = 0.
\end{cases}
\]
\ma{In particular,}
\[
f_e^-(\theta) \leq {c}_e\quad  \text{for all } e \in E \;\text{and } \theta\ge 0.
\]
Each flow particle entering edge $e$ at time $\theta$ thus experiences a delay of $q_e(\theta) = z_e(\theta)/{c}_e$ and then traverses the edge in time $\tau_e$, resulting in a final time of exiting $e$, i.e., $T_e(\theta) := \theta + q_e(\theta) + \tau_e$.
\purplecomment{Moreover, the queue dynamics imply that $T_e$ is nondecreasing. Hence, the FIFO (First-In-First-Out) property is respected on each edge, meaning that flow particles exit in the same order they enter: no particle can overtake another on the same edge.}

\subsection{The Nash Flow and Braess's Paradox}\label{sub:nashbp}

 In a Nash flow over time, the flow consists of infinitely many independent players (flow particles), each representing an infinitesimal unit of traffic. Each infinitesimal “particle” of traffic independently chooses its route so as to minimize its own arrival time at the sink $t$, given the evolving state of the network. A \purplecomment{\emph{Nash flow}} is a flow pattern \purplecomment{such that} no individual particle can reduce its arrival time at $t$ by unilaterally changing its route, assuming all other particles maintain their current paths. 
 
 \subsubsection{The Nash Flow}
 
In our definitions, we follow the refined notion of Nash flows from \cite{Cominetti2015}. For $v\in V$, let $\mathcal{P}_{s,v}$ denote the set of all $s$-$v$ paths in $G$. 
For \ma{any $\theta\ge0$}, let $l_v(\theta)$ denote the earliest possible arrival time at vertex $v$ for a flow particle entering $G$ from the source $s$ at time $\theta$. Then, for any edge $e = vw \in E$, the label functions $l_v(\cdot)$ satisfy
\begin{equation*}
    \begin{aligned}
        &l_s(\theta)=\theta,\\
        &l_w(\theta)=\min\{T_{e}(l_v(\theta))\mid e=vw\in \delta^-(w)\} \text{ for all } w\in V\setminus\{s\}.
    \end{aligned}
\end{equation*}
 \ma{A} path of consecutive edges \ma{$e$} attaining these minima \ma{is called a} \emph{dynamic shortest path} for particles entering the source at time $\theta$. We define the set of active edges for entering time $\theta$ as 
 $E'_{\theta} := \{ e=vw \in E \mid l_w(\theta) =T_{e}( l_v(\theta)) \}$, 
 \ma{consisting of} the edges that lie on dynamic shortest paths for entering time $\theta$. The set of edges with nonempty queues for entering time $\theta$ is defined as $E^*_{\theta} := \{  e=vw \in E \mid z_{e}(l_v(\theta)) >0 \}$.
 Let $G_\theta=(V, E_\theta')$ denote the dynamic shortest path network w.r.t. time $\theta$. \redcomment{Following \citet[Definition~1]{Cominetti2015}, a feasible flow over time is a \emph{Nash flow} if, for each $e=uv\in E$, we have $f_e^+(\vartheta)=0$ for almost every \ma{$\vartheta\in\{l_u(\theta')\mid \theta'\not\in\{\theta\mid e\in E'_\theta\}\}$. In other words,} 
 The label in this condition is that of the tail $u$: flow may enter an edge only while it is active, up to a null set of physical entry times.}
 \redcomment{By \citet[Proposition~2]{Cominetti2015}, the active and queued edge sets of a Nash flow also satisfy:} 
\begin{equation*}
    \begin{aligned}
        E'_{\theta} = \{ e=vw \in E \mid l_w(\theta) \ge l_v(\theta) + \tau_{e} \},\\
        E^*_{\theta} = \{ e=vw \in E \mid l_w(\theta) > l_v(\theta) + \tau_{e} \}.
    \end{aligned}
\end{equation*}

\redcomment{For piecewise-constant inflows, a right-constant dynamic equilibrium exists~\cite[Theorem~5]{Cominetti2015}. Moreover, the earliest-arrival label functions $(l_v)_{v\in V}$ are identical for all right-continuous dynamic equilibria~\cite[Theorem~6]{Cominetti2015}. For constant inflow, \citet{Olver2026} establish label uniqueness without the right-continuity assumption.}

\redcomment{Accordingly, throughout the nonuniform-inflow extensions we work with right-continuous equilibria and piecewise-constant inflows with locally finitely many changes. Nondecreasing always refers to the injection interval; the rate is zero before injection starts and after it ends. Rate identities are understood almost everywhere. These conventions include constant inflows, and do not assert uniqueness for arbitrary time-varying inflows.}

\redcomment{Following \citet[Definition~2 and Theorem~1]{Cominetti2015},} for a Nash flow $f$ and each edge $e = vw \in E$, define the cumulative inflow on edge $e$ up to time $l_v(\theta)$ as $x_e(\theta) := F_e^+(l_v(\theta))$, which has derivative
\[
x'_e(\theta) = \frac{\mathrm{d}x_e(\theta)}{\mathrm{d}\theta} = f_e^+(l_v(\theta)) \cdot l'_v(\theta).
\]
Actually, the vector $\mathbf{x'}(\theta) = (x'_e(\theta))_{e \in E}$ represents a static $s$-$t$ flow of value $\mu$, namely
\begin{equation}\label{char:static}
\sum_{e\in\delta^+(v)}x_e'(\theta)-\sum_{e\in\delta^-(v)}x_e'(\theta)=\begin{cases}
\mu & \text{if } v=s, \\
-\mu  & \text{if } v=t,\\
0  & \text{if } v\neq s,t.
\end{cases}   
\end{equation}
\redcomment{The following derivative relations are part of the normalized thin-flow characterization of Nash flows \cite[Theorem~2]{Cominetti2015}, extending the constant-inflow characterization of \citet{Koch2011}.}

\begin{theorem}\cite{Koch2011,Cominetti2015}\label{thm:eq}
     Consider a Nash flow $f$ and a time $\theta $ such that $x'_e(\theta)$ and $l'_v(\theta)$ exist for all $e\in E$ and $v\in V$. Then the static flow $\mathbf{x}'(\theta)$ satisfies:
    \begin{equation*}
        \begin{aligned}
            &l'_w(\theta)\le l'_v(\theta) &&\text{ for } e=vw\in E'_{\theta}\setminus E^*_{\theta} \text{ with } x'_{e}(\theta)=0;\\
            &l'_w(\theta)=\max \left\{l'_v(\theta), \frac{x'_e(\theta)}{c_e} \right\} &&\text{ for } e=vw\in E'_{\theta}\setminus E^*_{\theta}\text{ with }x'_{e}(\theta)>0;\\
            &l'_w(\theta)= \frac{x'_e(\theta)}{c_e}  &&\text{ for } e=vw\in  E^*_{\theta}.
        \end{aligned}
    \end{equation*}
\end{theorem}

Furthermore, for any $v\in V$, the derivative function $l'_v(\cdot)$ is unique across Nash flows ~\citep{Cominetti2015}, from which it follows that the arrival time functions $l_v(\theta)$ of the Nash flow are uniquely determined. 

\subsubsection{Braess's Paradox}
Following Macko et al.~\cite{Macko2013}, we examine Braess's paradox with respect to the maximum latency experienced by any flow particle in a Nash flow.  
For an instance $(G,s,t,{c}, \tau,\mu)$, the maximum latency cost of \purplecomment{any} Nash flow $f$ is defined as 
\[
SC(G,s,t,{c}, \tau,\mu):=\sup\{l_t(\theta)-\theta\mid \theta\ge 0\}.
\]
From the uniqueness of the arrival time functions $l_v(\theta)$, the above $SC$ of the Nash flows is well-defined and is exactly the social cost of the Nash flow $f$ defined by \citet{Macko2013}.
    For any subgraph $H$ of $G$ that still contains $s$ and $t$,  we slightly abuse notation by writing $(H, s, t, c,\tau,\mu)$ for the dynamic routing game instance obtained by restricting $(G,s,t,{c}, \tau,\mu)$ to $H$; that is, each edge of $H$  inherits the same capacity and transit time as in $G$.

\begin{definition}
    An instance $  (G, s, t, {c}, \tau, \mu)$ is said to admit \emph{Braess's Paradox} (\emph{BP})  if there exists an $s$-$t$ subgraph $H$ of $G$ such that the maximum latency cost of the Nash flow in the restricted instance is strictly lower, i.e.,
\(SC(H,s,t,{c}, \tau,\mu) < SC(G,s,t,{c}, \tau,\mu)\). An $s$-$t$ graph $G$ is said to admit \emph{BP} if there exist $c,\tau\in\mathbb{R}_+^E$ and $\mu>0$ such that  instance $  (G, s, t, {c}, \tau, \mu)$  admits BP.
\end{definition}
In other words, Braess’s paradox occurs when removing edges from the network leads to a \textit{strictly better} equilibrium outcome in terms of overall system performance.  This counterintuitive phenomenon illustrates that the presence of additional resources, e.g., network edges, can lead to less efficient equilibria under selfish routing.

The above definitions of Nash flow and Braess's paradox can be extended to the general scenario  $(G, s, t, {c}, \tau, \mu(\theta))$ with dynamic inflow rate $\mu(\theta)$, which may varies with time $\theta$, 
and define Braess's Paradox analogously for such time-dependent instances.


   A set of edges $ F$ in $G=(V,E)$ is called an $s$-$t$ cut if $V$ can be partitioned to two sets $S$ and $T$ such that $s\in S$, $t\in T$ and $F=\{uv\in E\mid u\in S\text{ and }v\in T\}$. The capacity of $F$ is defined as the total capacity  $\sum_{e\in F}c_e$ of its edges.
Let $c(G)$ denote the minimum capacity of an $s$-$t$ cut in $G$. In this paper, we only consider the case where $\mu\le {c}(G)$. Note that if $\mu>{c}(G)$, the maximum latency cost $SC(H,s,t,{c}, \tau,\mu)$  is infinite for all $s$-$t$ subgraphs $H$ of $G$, which excludes the presence of any BP.

\subsection{Conjecture}\label{conjecture}\label{sub:conj}

\citet{Macko2013} conjectured a necessary and sufficient condition for the occurrence of Braess’s paradox \purplecomment{(Conjecture 5.2)}. In this paper, we confirm the conjecture via rigorous proofs. 
Let $G$ and $H$ be two graphs. The operation of subdividing an edge $uv$ in $H$ once replaces $uv$ with a new vertex $w$ and two edges $uw$ and $wv$. A \emph{subdivision} of $H$ is the graph obtained from $H$ by sequentially subdividing edges (possibly none). We say that $G$ contains $H$ as a  \emph{topological minor} if $G$ has a subgraph isomorphic to a subdivision of $H$. 
A chain of parallel edges is a network obtained from an $s$-$t$ path by adding parallel edges (possibly none). A \emph{chain of parallel paths} is a subdivision of a chain of parallel edges. \purplecomment{The two parts of the following theorem are taken from \cite[Theorem~5.1 and Theorem~5.5]{Macko2013}, respectively.}
\begin{theorem}\cite{Macko2013}\label{thm:sufficient} 
Let $G$ be a single-source single-sink graph, and let $M_1,M_2,M_3,M_4$ be as depicted in Figure~\ref{fig:forbidden}. \redcomment{For a fixed positive constant inflow rate,}
\begin{enumerate}
    \item if $G$ contains at least one of $M_1, M_2, M_3$ as a topological minor, then $G$ admits BP.
    \item if $G$ contains none of $M_1, M_2, M_3, M_4$ as a topological minor, then $G$ is a chain of parallel paths, and admits no BP.
\end{enumerate}
\end{theorem}
\begin{figure}
    \centering
    \includegraphics[width=0.75\linewidth]{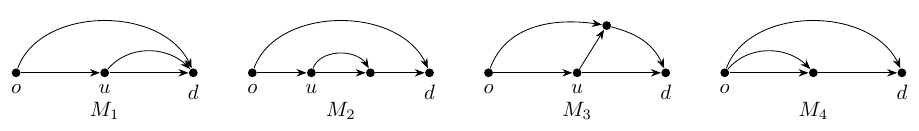}
     \caption{Forbidden minors for BP}\label{fig:forbidden}
\end{figure}

\begin{conjecture}\cite{Macko2013} \label{conj:BP}
\redcomment{For a fixed positive constant inflow rate,} a single-source single-sink graph $G$ admits BP if and only if $G$ contains at least one of  $M_1, M_2, M_3$ as a topological minor.    
\end{conjecture}

\section{Topological Characterizations}\label{sec:topo}
Let $G=(V,E)$ be a graph. If $w$ is a vertex or an edge of $G$, we often write $w\in G$, instead of $w\in V$ or $w\in E$. Graph  $G$ is \emph{connected} (resp.\ \emph{$2$-connected}) if its underlying undirected graph is connected (resp.\ 2-connected). A vertex in $G$ is a cut-vertex if its removal from $G$ leaves the graph unconnected.  A \emph{block} of $G$ is a maximal connected subgraph of $G$ without cut-vertices.  

\subsection{The $s$-$t$ Paradox}\label{subsec:paradox}
It is well-known that an irredundant graph is $s$-$t$ series-parallel if and only if its underlying undirected graph does not have two $s$-$t$ paths that pass a common edge in opposite directions. By the equivalent definition, one can prove that if an irredundant graph is not series-parallel, it must contain a so-called $s$-$t$ paradox as a subgraph (see, e.g., Theorem 3.5 of \cite{cdh2016}).
\begin{definition}\label{def:stparadox}
A {graph $H$} is called an \emph{$s$-$t$ paradox} if $H=P_1\cup P_2\cup P_3$ is the union of three paths $P_1$, $P_2$ and $P_3$ with the following properties:
\begin{itemize}
\item $P_1$ is an $s$-$t$ path   going through distinct vertices $a,u,v,b$ in this order;
\item $P_2$ is an $a$-$v$ path   with $V(P_2)\cap V(P_1)=\{a,v\}$;
\item $P_3$ is a $u$-$b$ path   with $V(P_3)\cap V(P_1)=\{u,b\}$  and $V(P_3)\cap V(P_2)=\emptyset$.
\end{itemize}
\end{definition}
\begin{figure}[h!]
    \centering
    \includegraphics[width=0.5\linewidth]{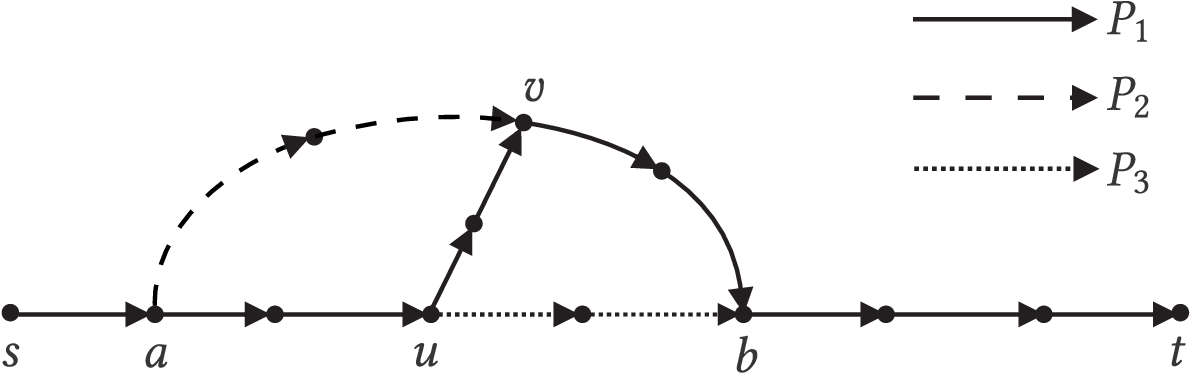}
    \caption{An $s$-$t$ paradox}
    \label{fig:st-paradox}
\end{figure}
\begin{lemma}\cite{cdh2016}\label{lem:st-paradox}
  Let $G$ be an irredundant $s$-$t$ graph. Then $G$ is $s$-$t$ series-parallel if and only if $G$ does not contain any $s$-$t$ paradox.
\end{lemma}
Clearly, the occurrence of the $s$-$t$ paradox implies that the graph contains $M_3$ as a topological minor. Lemma~\ref{lem:st-paradox} gives the following immediate corollary.

\begin{corollary}\label{cor:series-parallel}
    If an irredundant graph contains no $M_3$ as a topological minor, then it is series-parallel.
\end{corollary}

{It is evident that $s$-$t$ pseudo-arborescences must be irredundant. Moreover, they do not contain any $s$-$t$ paradox, because the paradox possesses a vertex $u\notin\{s,t\}$ and at least two different $u$-$t$ paths (see Figure~\ref{fig:st-paradox} for an illustration). It follows from Lemma~\ref{lem:st-paradox} that bipolar pseudo-arborescences (BPAs) are all series-parallel.}

\subsection{Equivalent Graphical Description}\label{sub:eqi}
It can be shown that  an irredundant graph containing no $M_3$ as a topological minor must be series-parallel (see Corollary~\ref{cor:series-parallel} in \purplecomment{Section} \ref{subsec:paradox}).
\begin{definition}
    A graph $G$ with source $s$ and sink $t$ is \emph{$s$-$t$ series-parallel} or simply \emph{series-parallel}  if one of the following holds:
    \begin{itemize}
        \item $G$ consists of a single edge $st$. 
        \item $G$ can be obtained by connecting two smaller $s_i$-$t_i$ series-parallel graphs $G_i$, $i=1,2$ \emph{in series}:  merging $t_1$ with $s_2$ and renaming $s_1$ as $s$ and $t_2$ as $t$. 
        \item $G$ can be obtained by connecting two smaller $s_i$-$t_i$ series-parallel graphs $G_i$, $i=1,2$ \emph{in parallel}: merging $s_1$ and $s_2$ into $s$ and  merging $t_1$ and $t_2$ into $t$.
    \end{itemize}
\end{definition}

\begin{definition}\label{def;bpa}
    An $s$-$t$ graph is an \emph{$s$-$t$ pseudo-arborescence} or  a \emph{bipolar pseudo-arborescence} (\emph{BPA}) if $|\delta^-(s)|=0$,  $|\delta^+(t)|=0$, and $|\delta^+(v)|=1$ for any other vertex $v$. 
\end{definition}

An $s$-$t$ pseudo-arborescence is, in essence, an in-arborescence rooted at sink $t$ in which all leaf vertices are merged into a single source vertex $s$. The unique source $s$ and sink $t$ in the BPA are collectively called its \emph{polar vertices}. Let $G=(V,E)$ be an  $s$-$t$ pseudo-arborescence, and let $\alpha,\beta$ be two vertices or two edges in $G$. We call $\alpha$ a descendant of $\beta$, and equivalently write $\alpha\prec\beta$, if there is an $s$-$t$ path in $G$ that visits $\alpha$ before $\beta$. By the arborescence structure, for any $\alpha$ in $ V\setminus\{s,t\}$ or $E\setminus\delta^-(t)$, the vertex set or edge set of the unique path starting from $\alpha$ and ending at $t$ consists of $\alpha$ and all its ancestors. We say that $\alpha$ and $\beta$ are \emph{parallel} if $\alpha\not\prec\beta$ and $\beta\not\prec\alpha$. 

\begin{remark} \label{re:1path}
    If $G$ is an $s$-$t$ pseudo-arborescence, then it is irredundant, and for any vertex $v\in V\setminus \{s\}$, the $v$-$t$ path is unique, which implies that $G$ is series-parallel. 
\end{remark}

\begin{definition}
    A graph is called a \emph{chain of bipolar pseudo-arborescences} (\emph{chain of BPAs}) if it is obtained by sequentially connecting bipolar pseudo-arborescences in series. \redcomment{(see Figure~\ref{fig:BPAs}).}
\end{definition}
    
\begin{figure}[H]
    \centering
    \includegraphics[width=0.75\linewidth]{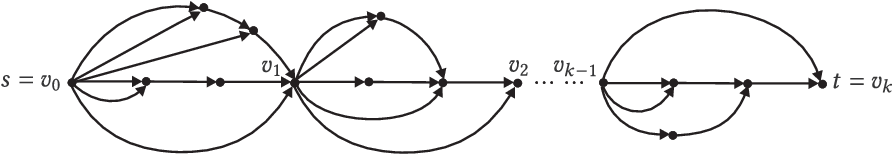}
    \caption{A chain of BPAs. \redcomment{Inside each block, merged branches share a unique suffix; new splits occur only at the polar vertices joining blocks.}}
    \label{fig:BPAs}
\end{figure}

 By Remark \ref{re:1path},  all chains of BPAs form a special class of series-parallel graphs. Clearly, a series-parallel graph is a chain of BPAs if and only if all its blocks are BPAs.

\begin{lemma}\label{lem:topology}
    An irredundant graph is a chain of BPAs if and only if it contains none of $M_1, M_2, M_3$ as a topological minor.
\end{lemma}
\begin{proof}
    The necessity is straightforward. Suppose that $G$ is a chain of BPAs. If  $G$ contains some $M_i$ ($i\in\{1,2,3\}$) as a topological minor, then so does one of its block, which is a BPA. It follows that this BPA  has a non-polar vertex $u$ with out-degree 2 or more, a contradiction to Definition~\ref{def;bpa}. 

    To see the sufficiency, suppose that $G$ contains none of $M_1,M_2,M_3$ as a topological minor. It follows from Corollary~\ref{cor:series-parallel} that $G$ is series-parallel. Assume for a contradiction that $G$ is not a chain of BPAs. Then $G$ has a block $B$ that is not BPA. Because $G$ is series-parallel, so is its block $B$. Let $o$ and $d$ denote the source and sink of $B$, respectively. Note that every vertex in $V(B)\setminus\{o,d\}$ has out-degree at least 1. Therefore, there exists $u\in V(B)\setminus\{o,d\}$ with $|\delta^+_B(u)|\ge2$. It follows from $|V(G)|\ge|\{o,d,u\}|=3$ that $B$ is 2-connected. 
    \purplecomment{By the $o$-$d$ series-parallel structure, there exists an $o$-$d$ path $P_1$ containing $u$. Moreover, by the 2-connectivity of $B$, there exists an $o$-$d$ path $P_2$ avoiding $u$; otherwise, $u$ would be a cut vertex of $B$.}
    By $|\delta^+_B(u)|\ge2$, there exists $uw\in\delta^+_B(u)\setminus E(P_1)$. Let $P_3$ be a shortest path from $w$ to a vertex in $P_1\cup P_2$. By the acyclic property of series-parallel graphs, it is easy to check that $P_1\cup P_2\cup uw\cup P_3$ is a subdivision of $M_1$ or $M_2$ or $M_3$, a contradiction.\qed 
    \end{proof}

\subsection{Network Decomposition}\label{sub:nde}
For any edge subset $F$ of a graph $G$, let $G[F]$ denote the subgraph of $G$ induced by $F$, where the edge set is $F$, and the vertex set consists of the end-vertices of edges in $F$. For simplicity, we usually identify a graph with its edge set when no ambiguity arises. In other words, we typically use $F$ to represent $G[F]$ or vice versa.

Given an instance $(G,s,t,c,\tau,\mu)$ or $(G,\mu)$ for short, in this subsection, we always assume $\mu\le c (G)$ and $G$ is series-parallel. For any $s$-$t$ subgraph $H$ of $G$, we define its (source-sink) \emph{capacity} $c(H)$  as the minimum capacity of an $s$-$t$ cut in it, and its (source-sink) \emph{diameter} $\tau(H)$ as the minimum length of an $s$-$t$ path. In particular,  if $P$ is an $s$-$t$ path, then $c(P)=\min_{e\in P}c_e$ is its minimum edge capacity, and $\tau(P)=\sum_{e\in P}\tau_e$ equals its length. The \emph{shortest-path subgraph} of $(G,\tau)$ or $G$ (if $\tau$ is clear from the context) is the subgraph of $G$ that consists of all shortest $s$-$t$ paths in $G$. For notational simplicity, for any function $g$ defined on $E(G)$ (such as $c$ and $\tau$), when considering its restriction to the edge set of a subgraph of $G$, we usually use the same symbol $g$ to denote this restriction. 

Unlike the dynamic flows discussed in Section 2, the $s$-$t$ flows mentioned in this section are all static flows in the classical edge-capacitated network setting. These static flows serve merely as technical tools to help us characterize the network structure, preparing for our subsequent study of dynamic flows (the main subject of this paper). \purplecomment{Here, we denote by $\mathbb R_+$ and $\mathbb R_{++}$ the sets of nonnegative and positive real numbers, respectively.} For any $s$-$t$ flow $\eta$ in $G=(V.E)$ with edge capacity $c\in\mathbb R_{++}^E$, we consider it both as an edge flow $\eta \in \mathbb{R}_+^E$, which assigns $\eta_e \in [0, c_e]$ to every edge $e$ and satisfies flow conservation conditions, and as a path flow $\eta \in \mathbb{R}_+^{\mathcal P_{s,t}}$, which assigns a non-negative $\eta_P$ to every $s$-$t$ path $P$ and satisfies capacity constraints: $\sum_{P:e\in P}\eta_P\le c_e$, $e\in E$. The (flow) value of $\eta$ is denoted by $|\eta|=\sum_{P\in\mathcal{P}_{s,t}}\eta_P$. 

Next, we decompose $G$ by iteratively ``removing'' an $s$-$t$ flow in the current shortest-path subgraph, until the total flow value reaches $\mu$. Each $s$-$t$ flow removed is a maximum flow within the current shortest-path subgraph. The only possible exception is the last iteration, where the flow may not saturate the minimum cut, as the target flow amount $\mu$ has already been met. The formal process is presented in the Algorithm~\ref{alg:flow} below.

 \begin{algorithm}
    \caption{Network Decomposition by Shortest-Path Flows}\label{alg:flow}
    \KwIn{instance $(G,s,t,c,\tau,\mu)$, where $G=(V,E)$ is series-parallel and $0<\mu\le c(G)$}
    \KwOut{a series of $(\xi^{(i)},c^{(i)},\tau^{(i)})$, $i=1,2,\ldots$}
   $\mu^{(1)}\gets\mu$;  $ G^{(1)} \gets G$; $i\gets 0$\;
    $G^{\xi,1}\gets G$; $c^{\xi,1}\gets c$\tcp*[r]{\tiny We record $G^{\xi,1}$ and $c^{\xi,1}$ for theoretical proofs only}
  \Repeat{\redcomment{$\mu^{(i+1)}= 0$}}
    {  $i \leftarrow i+1$\;
       $H^{(i)}\gets$ the shortest-path graph of $(G^{(i)},\tau)$, $c^{(i)}\gets c(H^{(i)})$; $\tau^{(i)}\gets\tau(H^{(i)})$\;
        $\xi^{(i)}\gets $ an $s$-$t$ flow in $(H^{(i)},c)$ with value \redcomment{$|\xi^{(i)}|=\min\{\mu^{(i)},c(H^{(i)})\}$}\;
         $\mu^{(i+1)}\gets\mu^{(i)}-|\xi^{(i)}|$\;
         { $c_e\gets c_e-\xi^{(i)}_e$ for all $e\in G^{(i)}$ \tcp*[r]{\tiny Remove capacities occupied by $\xi^{(i)}$}
         \purplecomment{$G^{(i+1)}\gets$ the graph induced by $\{e\in E\mid c_e>0\}$\;}
         \redcomment{$G^{\xi,i+1}\gets G^{(i+1)}$; $c^{\xi,i+1}\gets c|_{E(G^{(i+1)})}$\;}
           }
      }
   $m\gets i$\;      
    \Return{$\xi^{(i)}$ and $(c^{(i)},\tau^{(i)})$, $i=1,\ldots,m$}
  \end{algorithm}
 
  In Algorithm~\ref{alg:flow}, $G^{\xi,m+1}$ is either a graph or an empty set. Possibly $G^{\xi,m+1}$ contains no $s$-$t$ path. 

We call $\xi=(\xi^{(i)}:i=1,\ldots,m)$ a \emph{shortest-path-flow decomposition} or \emph{SPF decomposition} of instance $(G,s,t,c,\tau,\mu)$ with \emph{dimension} $\mathrm{dim}(\xi)=m$, and call 
$(c^{(i)},\tau^{(i)})_{i=1}^m$ the \emph{capacity-diameter slicing}  of the instance associated with $\xi$. For each $i=1,\ldots,m$, flow $\xi^{(i)}$ goes along some $s$-$t$ paths in the current shortest-path graph $H^{(i)}$; all these paths are shortest and have the identical length $\tau^{(i)}=\tau(H^{(i)})$. For ease of future reference, we define $\tau(\xi^{(i)}) = \tau^{(i)}$ by a slight abuse of notation. Clearly, the diameter sequence is non-decreasing: 
$\tau(\xi^{(1)})\le \tau(\xi^{(2)})\le\cdots\le\tau(\xi^{(m)})\le\tau(G^{\xi,m+1})$.



\subsection{Topological Results}\label{sub:topore}


\redcomment{The next results through Lemma~\ref{lem:monotone} hold for general series-parallel graphs: augmentation first gives a well-defined decomposition, and numerical slice invariance then yields monotonicity of the terminal SPF diameter. Starting with Lemma~\ref{lem:rightmost}, we specialize to BPAs to obtain a stronger invariant, namely a canonical rightmost cut. This distinction separates what follows from series-parallel structure alone from what is needed in the dynamic analysis.}

In this subsection, we \emph{assume that  $G=(V,E)$ is an $s$-$t$ series-parallel graph} with edge capacity $c\in\mathbb R^E_{++}$ and length $\tau\in\mathbb R^E_+$.

An $s$-$t$ cut in $G$ is \emph{minimal} if none of its proper subsets is an $s$-$t$ cut. For any path $P$ in $G$, the subpath of $P$ with starting vertex $u$ and ending vertex $v$ is written as $P[u,v]$; moreover, let $P(u,v]$, $P[u,v)$ and $P(u,v)$ be obtained from $P[u,v]$ by deleting $u$, $v$ and $\{u,v\}$, respectively.
\begin{lemma}\label{lem:path-cut}
   {For any $s$-$t$ path $P$ and minimal $s$-$t$ cut $F$ in $G$, it holds that $|P\cap F|=1$.}
\end{lemma}
\begin{proof}
{Take an arbitrary $s$-$t$ path $P$ and a minimal $s$-$t$ cut $F$ in $G$. Clearly $|P\cap F|\ge1$. Assume to the contrary that  $P$ goes through two edges $e_1$ and $e_2$ in this order and $e_1,e_2\in F$. Since $F$ is minimal, there is $s$-$t$ path $P_i$ such that $P_i\cap F=\{e_i\}$ for $i=1,2$. It is not hard to see that $P_1\cup P_2\cup P$ contains an $s$-$t$ paradox, a contradiction to Lemma~\ref{lem:st-paradox}.}

{Indeed, since $G$ is acyclic, there are common vertices $a,b$ of $P_1$ and $P_2$ such that $P_i[a,b]$ contains $e_i=a_ib_i$, $i=1,2$ and $P_2(a,b)$ is vertex-disjoint from $P_1$. Let $P[u,v]$ be a shortest subpath in $P[b_1,a_2]$ such that $u\in P_1[b_1,t]$ and $v\in P_2[a,b)\cup P_1[s,a_1]$ (note that  $P[b_1,a_2]$ satisfies this condition). The minimality of $P[u,v]$ implies that $P(u,v)$ is disjoint from $P_1\cup P_2[u,v]$. The acyclic property enforces $u\not\in P_1[b,t]$ and $v\not\in P_1[s,a_1]$, It follows from $u\in P_1[b_1,b)$ and $v\in P_2(a,b)$ that $P_1\cup P_2[a,b]\cup P[u,v]$ is an $s$-$t$ paradox in $G$.}
\qed
\end{proof}

By Lemma~\ref{lem:path-cut}, a subset of edges $F\subseteq E$ is a minimal $s$-$t$ cut in the series-parallel graph $G$ if and only if $V$ can be partitioned to $S$ and $T$ such that $s\in S$, $t\in T$, all edges between $S$ and $T$ are directed from $S$ and $T$, and these edges are exactly those in $F$. 

Let $\eta$ be an $s$-$t$ flow in $(G,c)$. By \emph{augmenting} $\eta$, we mean increasing $\eta_P$ for some $s$-$t$ paths $P$ without violating edge capacities. We say that $\eta$ \emph{saturates} an $s$-$t$ cut if $\eta_e=c_e$ for every edge $e$ in the cut. Lemma~\ref{lem:path-cut} yields two crucial properties for series-parallel graphs:
\begin{itemize}
    \item \emph{Augmentation property}: any $s$-$t$ flow can be augmented to a maximum $s$-$t$ flow in $(G,c)$, whose value equals $c(G)$, the capacity of the graph.
    \item \emph{Saturation property}: any $s$-$t$ flow that saturates an $s$-$t$ cut must be maximum.
\end{itemize}
These two properties will be used repeatedly (often implicitly) in our subsequent analysis. They are equivalent because an $s$-$t$ flow cannot be strictly augmented if and only if it saturates some $s$-$t$ cut.

\redcomment{Here augmentation preserves every existing path-flow component of an arbitrary feasible $\eta$: it only adds nonnegative path flows. This is the structural form we need. The forward-only greedy method of \citet{BEIN1985} instead starts from zero and successively augments a cheapest path to solve minimum-cost flow problems. We do not claim forward-only augmentation itself as a new algorithmic idea.}

Each $s$-$t$ flow $\xi^{(i)}$ computed in Algorithm~\ref{alg:flow} can be considered as a path flow that assigns a nonnegative value $\xi^{(i)}_P$ to each $s$-$t$ path $P$ in the shortest-path graph of $G^{\xi,i}$ such that the sum of all these $\xi^{(i)}_P$'s equals the flow value $|\xi^{(i)}|$. By the definitions of $G^{\xi,i}$ and $c^{\xi,i}$, $i=1,2,\ldots$, Lemma~\ref{lem:path-cut} guarantees that Algorithm~\ref{alg:flow} terminates with finite dimension $\mathrm{dim}(\xi)=m$, and the following statement holds.

\begin{corollary}\label{lem:path}
   \redcomment{For each $i=1,\ldots,m$,
\[
\mu^{(i+1)}=\mu^{(i)}-|\xi^{(i)}|,\qquad
c^{\xi,i+1}(G^{\xi,i+1})=c^{\xi,i}(G^{\xi,i})-|\xi^{(i)}|.
\]
Moreover, $\sum_{i=1}^m|\xi^{(i)}|=\mu$. 
}
\end{corollary}




By Corollary~\ref{lem:path}, we can also regard the SPF decomposition $\xi=(\xi^{(i)}:i=1,\ldots,m)$ returned by Algorithm~\ref{alg:flow} as an $s$-$t$ flow in $G$, with the understanding that $\xi_e=\sum_{i:e\in H^{(i)}}\xi^{(i)}_e$ if $e\in\cup_{i=1}^mH^{(i)}$ and $\xi_e=0$ otherwise. The flow value of $\xi$ is $|\xi|=\sum_{i=1}^m|\xi^{(i)}|=\mu$.

Given any $s$-$t$ flow $\eta$ of $(G,c)$,  the \emph{support} of $\eta$ is the subgraph induced by $\{e\in E\mid\eta_e>0\}$. When $\eta$ is viewed as the edge capacity function on its support, the source-sink capacity of the support is also called the \emph{capacity} of  $\eta$, which equals $|\eta|$ by Lemma~\ref{lem:path-cut}. The following fact is still a simple corollary of Lemma \ref{lem:path-cut}, because every $s$-$t$ path in $G$ meets every minimal $s$-$t$ cut in any subgraph of $G$ in at most one edge.
\begin{corollary}\label{cor:subgraph}
    Let $H$ be an $s$-$t$ subgraph of $G$. If $c(H)>|\eta|$, then $\{e\in E(H)\mid c_e>\eta_e\}$ contains at least an $s$-$t$ path.
\end{corollary}

    
For any capacity-diameter slicing $(c^{(i)},\tau^{(i)})_{i=1}^m$ of instance $(G,s,t,c,\tau,\mu)$, with $H$ denoting the shortest-path subgraph of $G$, we have invariants $c^{(1)}=c(H)$ and $\tau^{(1)}=\tau(H)=\tau(G)$ despite the general non-uniqueness of SPF decompositions. Our goal is to prove that the entire slicing (including the dimension $m$) is an invariant of the instance, independent of the specific form of the flow decomposition. 
For notational convenience, 
if $H$ is an $s'$-$t'$ subgraph of $G$, we often write $(H,s',t',c,\tau,\mu)$ or simply $(H,\mu)$ instead of $(H,s',t',c|_{E(H)},\tau|_{E(H)},\mu)$.

\begin{lemma}\label{lem:flow-invariant}
Let $\xi=(\xi^{(i)}:i=1,\ldots,m)$ and $\zeta=(\zeta^{(j)}:j=1,\ldots,n)$ be two SPF decompositions of $(G,s,t,\tau,{c},\mu)$.  The following hold.
\begin{itemize}
    \item[(i)] If  $\mu=c(G)$, then neither $G^{\xi,m+1}$ nor $G^{\zeta,n+1}$ contains an $s$-$t$ path.
  \item[(ii)]     If $\mu<c(G)$, then  $\tau(G^{\xi,m+1})=\tau(G^{\zeta,n+1})$, and $c^{\xi,m+1}(K)=c^{\zeta,n+1}(Z)$, where $K$ and $Z$ are the shortest-path subgraphs of $(G^{\xi,m+1},\tau)$ and $(G^{\zeta,n+1},\tau)$, respectively.
\end{itemize} 
\end{lemma}

\begin{proof}
{Statement (i) is instant from Corollary \ref{lem:path}. We prove (ii) by the inductive structure of series-parallel graphs. The base case, where $G$ consists of a single edge $st$, is trivial. Assume that $G$ is a series- or parallel-connection of two smaller series-parallel graphs $G_1$ and $G_2$, for which the lemma holds. }


\paragraph{Case 1: series-connection} {Every shortest $s$-$t$ path in $G$ is a series-connection of a shortest source-sink path in $G_1$ and a shortest source-sink path in $G_2$, and vice versa. For each $i=1,2,\ldots,m$, the shortest-path graph of $G^{\xi,i}$ is the series-connection of the shortest-path graphs of $G^{\xi,i}\cap G_1$ and $G^{\xi,i}\cap G_2$, and vice versa. By Lemma~\ref{lem:path-cut}, removing a source-sink flow from the network reduces the source-sink capacity by the value of the flow; equivalently, every source-sink flow can be augmented to be a maximum one (i.e., the augmentation property). It follows that, for $h=1,2$, there is an SPF decomposition $\alpha(h)$ of $(G_h,\mu)$ with dimension $m(h)$ such that $\alpha(h)_e=\xi_e$ for all $e\in G_h$, where $\alpha(h)=(\alpha(h)^{(k)}:k=1,\ldots,m(h))$ is a splitting-and-regrouping of the path flow restriction of $\xi$  to $G_h$. Therefore 
\[G^{\xi,m+1}\cap G_h=G_h^{\alpha(h),m(h)+1}\text{ and }c^{\xi,m+1}_e=c^{\alpha(h),m(h)+1}_e\text{ for all }e\in G^{\xi,m+1}\cap G_h.\]
Similarly, for $h=1,2$, there is an SPF decomposition $\beta(h)$ of $(G_h,\mu)$ with dimension $n(h)$ such that 
\[G^{\zeta,n+1}\cap G_h=G_h^{\beta(h),n(h)+1}\text{ and }c^{\zeta,n+1}_e=c^{\beta(h),n(h)+1}_e\text{ for all }e\in G^{\zeta,n+1}\cap G_h.\]
Since $\mu<c(G)=\min\{c(G_1),c(G_2)\}$, by induction hypothesis, for $h=1,2$, we have 
\[\tau(G_h^{\alpha(h),m(h)+1})=\tau(G_h^{\beta(h),n(h)+1})\text{ and }c^{\alpha(h),m(h)+1}(K_h)=c^{\beta(h),n(h)+1}(Z_h),\] where $K_h$ and $Z_h$ are the shortest-path subgraphs of $G_h^{\alpha(h),m(h)+1}=G^{\xi,n+1}\cap G_h$ and $G_h^{\beta(h),n(h)+1}=G^{\zeta,n+1}\cap G_h$, respectively. Thus, from 
$\tau(G^{\xi,m+1})=\tau(G^{\xi,m+1}\cap G_1)+\tau(G^{\xi,m+1}\cap G_2)=\tau(G^{\alpha(1),m(1)+1}_1)+\tau(G^{\alpha(2),m(2)+1}_2)$,
we derive
\begin{align*}
    \tau(G^{\xi,m+1})&=\tau(G^{\beta(1),n(1)+1}_1)+\tau(G^{\beta(2),n(2)+1}_2)=\tau(G^{\zeta,n+1}\cap G_1)+\tau(G^{\zeta,n+1}\cap G_2)=\tau(G^{\zeta,n+1}).
\end{align*}
Note that $K$ is series-connection of $K_1$ and $K_2$, and $Z$ is series-connection of $Z_1$ and $Z_2$. It follows from $
 c^{\xi,m+1}(K)=\min\{c^{\xi,m+1}(K_1),c^{\xi,m+1}(K_2)\} =\min\{c^{\alpha(1),m(1)+1}(K_1),c^{\alpha(2),m(2)+1}(K_2)\}$
that
\begin{align*}
 c^{\xi,m+1}(K)=\min\{c^{\beta(1),n(1)+1}(Z_1),c^{\beta(2),n(2)+1}(Z_2)=\min\{c^{\zeta,n+1}(Z_1),c^{\zeta,n+1}(Z_2)\}=c^{\zeta,n+1}(Z).
\end{align*}}

\paragraph{Case 2: parallel-connection} For $h=1,2$, let $\alpha(h)=(\alpha(h)^{(k)}:k=1,\ldots,m(h))$ be the restriction of $\xi$ to $G_h$.  Then $|\alpha(1)|+|\alpha(2)|=\mu$ and $\alpha(h)$ is an SPF decomposition of $(G_h,|\alpha(h)|)$ for $h=1,2$. {In case $\xi$'s support is entirely contained in some $G_h$, we treat $\alpha(3-h)$ with $|\alpha(3-h)|=0$ as a null SPF decomposition of $(G_{3-h},0)$, which corresponds to a zero $s$-$t$ flow in $G_h$. Symmetrically, let $\beta(h)=(\beta(h)^{(k)}:k=1,\ldots,n(h))$ be the restriction of $\zeta$ to $G_h$ for $h=1,2$. Then $|\beta(1)|+|\beta(2)|=\mu$ and $\beta(h)$ is an SPF decomposition of $(G_h,|\beta(h)|)$ for $h=1,2$. In case $\zeta$'s support is entirely contained in some $G_h$, we treat $\beta(3-h)$ with $|\beta(3-h)|=0$ as a null SPF decomposition of $(G_{3-h},0)$, which corresponds to a zero $s$-$t$ flow in $G_h$.

We first prove $\tau(G^{\xi,m+1})=\tau(G^{\zeta,n+1})$.
If $|\alpha(1)|=|\beta(1)|$, equivalently $|\alpha(2)|=|\beta(2)|$, then it is routine to check that the induction hypothesis on $(G_1,|\alpha(1)|)$ and $(G_2,|\alpha(2)|)$ yields $\tau(G^{\xi,m+1})=\tau(G^{\zeta,n+1})$. So we assume without loss of generality that $|\alpha(1)|>|\beta(1)|$, which implies $|\alpha(2)|<|\beta(2)|$. Let $\ell(\alpha)$ denote the length of the longest $s$-$t$ path in the support of $\alpha(1)$, and let $\ell(\beta)$ denote the length of the longest $s$-$t$ path in the support of $\beta(2)$.

By Algorithm~\ref{alg:flow}, the support of $\alpha(1)$ contains at least one $s$-$t$ path in each shortest-path subgraph of $G^{\xi,i}$, $i=1,\ldots,m$. Note that all $s$-$t$ paths in the same shortest-path subgraph have the same length. The property ensures that  $\alpha(1)$'s support contains at least one $s$-$t$ path of length $\ell(\alpha)=\tau(G^{\xi,m})$. Similarly, $\beta(2)$'s support contains at least one $s$-$t$ path of length $\ell(\beta)=\tau(G^{\zeta,n})$.}

By Corollary~\ref{cor:subgraph}, we deduce from $|\alpha(1)|>|\beta(1)|$ that the support of $\alpha(1)$ contains an $s$-$t$ path $A$ with  $c_e>\beta(1)_e$ for every edge $e\in A$. It follows that $G_1^{\beta(1),n(1)+1}=G^{\zeta,n+1}\cap G_1$ contains $A$ with length  \[\tau(A)\in[\tau(G^{\zeta,n+1}\cap G_1),\ell(\alpha)].\]
 Symmetrically, $|\beta(2)|>|\alpha(2)|$ implies that the support of $\beta(2)$ contains an $s$-$t$ path $B$ with  $c_e>\alpha(2)_e$ for every edge $e\in B$, and $G_2^{\alpha(2),m(2)+1}=G^{\xi,m+1}\cap G_2$ contains $B$ with length 
 \[\tau(B)\in[\tau(G^{\xi,m+1}\cap G_2),\ell(\beta)].\]
 Recall that the shortest-path-flow decomposition by Algorithm~\ref{alg:flow} prefers shorter paths to longer ones, saturating the minimum cuts of the current shortest-path subgraph before sending flow along longer paths in the next shortest-path subgraph. We have $\ell(\alpha)\le\tau(B)\le\ell(\beta)$ and $\ell(\beta)\le\tau(A)\le\ell(\alpha)$, saying that  $\ell(\alpha)$, $\tau(B)$, $\tau(A)$, $\ell(\beta)$ are equal to the same value, written as $\tau^*$.
 
 Moreover, recalling the non-decreasing property of the diameter sequence returned by Algorithm~\ref{alg:flow}, we have $\tau(G^{\xi,m+1}\cap G_1)\ge\tau(G^{\xi,m+1})\ge\tau(G^{\xi,m})=\ell(\alpha)=\tau^*$ (in case $G^{\xi,m+1}\subseteq G_2$, we define $\tau(G^{\xi,m+1}\cap G_1):=\infty)$ and $\tau^*=\ell(\alpha)\le\tau(G^{\xi,m+1}\cap G_2)\le\tau(B)=\tau^*$, which enforces
 \[\tau(G^{\xi,m+1})=\min\{\tau(G^{\xi,m+1}\cap G_1),\tau(G^{\xi,m+1}\cap G_2)\}=\tau(G^{\xi,m+1}\cap G_2)=\tau^*.\]
 Similarly,  $\tau(G^{\zeta,n+1}\cap G_2)\ge\tau(G^{\zeta,n+1})\ge\tau(G^{\zeta,n})=\ell(\beta)=\tau^*$ (where $\tau(G^{\zeta,n+1}\cap G_2):=\infty$ in case $G^{\zeta,n+1}\subseteq G_1)$) and $\tau^*=\ell(\beta)\le\tau(G^{\zeta,n+1}\cap G_1)\le\tau(A)=\tau^*$ enforce 
 \[\tau(G^{\zeta,n+1})=\min\{\tau(G^{\zeta,n+1}\cap G_1),\tau(G^{\zeta,n+1}\cap G_2)\}=\tau(G^{\zeta,n+1}\cap G_1)=\tau^*,\]
giving $\tau(G^{\xi,m+1})=\tau^*=\tau(G^{\zeta,n+1})$ as desired.

Let $H$ be the subgraph of $G$ consisting of all $s$-$t$ paths of length at most $\tau(G^{\xi,m+1})=\tau(G^{\zeta,n+1})=\tau(K)=\tau(Z)$. Note that $\xi$ and $\zeta$ also serve as SPF decompositions of $(H,\mu)$, and that $K\subseteq H^{\xi,m+1}$ and $Z\subseteq H^{\zeta,n+1}$. Because $K$ is the shortest-path subgraph of $G^{\xi,m+1}$, it is the shortest-path graph of $H^{\xi,m+1}(\subseteq G^{\xi,m+1})$. By the definition of $H$, the longest $s$-$t$ path in $H^{\xi,m+1}$ (if any) has length at most $\tau(K)$. Thus all $s$-$t$ paths in $H^{\xi,m+1}$ have identical length $\tau(K)$, and these paths are all contained in $K$, thereby constituting $K$. Hence $c^{\xi,m+1}(K)=c^{\xi,m+1}(H^{\xi,m+1})$. Similarly, $c^{\zeta,n+1}(Z)=c^{\zeta,n+1}(H^{\zeta,n+1})$. It follows from Lemma~\ref{lem:path-cut} that
\[c^{\xi,m+1}(K)=c^{\xi,m+1}(H^{\xi,m+1})=c(H)-\mu= c^{\zeta,n+1}(H^{\zeta,n+1})=c^{\zeta,n+1}(Z),\]
which completes the proof.
\qed
\end{proof}

\begin{corollary}\label{cor:step}
\redcomment{Let $\xi$ and $\zeta$ be two SPF decompositions of $(G,s,t,c,\tau,\mu)$. Suppose both have completed $j$ iterations and have routed the same amount
$q=\sum_{h=1}^j|\xi^{(h)}|=\sum_{h=1}^j|\zeta^{(h)}|$, where $q=0$ if $j=0$.
If $q=\mu$, both decompositions terminate; if also $q=c(G)$, neither residual network has an $s$-$t$ path. If $q<\mu$, both have a next iteration, and their residual shortest-path lengths and capacities coincide. Consequently, their next slice values $|\xi^{(j+1)}|$ and $|\zeta^{(j+1)}|$ coincide.}
\end{corollary}
\begin{proof}
\redcomment{For $j=0$, the two residual instances are identical. For $j>0$, the two prefixes are SPF decompositions routing $q$ units. Apply Lemma~\ref{lem:flow-invariant}. If $q<\mu$, the next slice value is the minimum of the common residual shortest-path capacity and the common remaining amount $\mu-q$.}
\qed
\end{proof}

The selection of $s$-$t$ flow $\xi^{(i)}$ in Algorithm~\ref{alg:flow} is arbitrary, which may result in different SPF decompositions for the same instance. However, as the following corollary states, the capacity-diameter slicings associated with these SPF decompositions are unique.
\begin{corollary}\label{lem:uni}\label{cor:uni}
    For any instance $(G,s,t,c,\tau,\mu)$, the capacity-diameter slicing $(c^{(i)},\tau^{(i)})_{i=1}^m$ output by Algorithm~\ref{alg:flow}  is unique.
\end{corollary}

By Corollary \ref{cor:step}, the proof of Corollary~\ref{lem:uni} is straightforward. Given the unique capacity-diameter slicing $(c^{(i)},\tau^{(i)})_{i=1}^m$ of instance $(G,s,t,c,\tau,\mu)$, we define 
\[\tau(G, c, \mu) := \tau^{(m)}\] 
as the length of the longest $s$-$t$ paths that carry some flow; this length is identical across all SPF decompositions of the instance.  For any $s$-$t$ subgraph $H$ of $G$, we write $\tau(H,c,\mu)$ instead of $\tau|_{E(H)}(H,c|_{E(H)},\mu)$ for notational convenience.

\begin{lemma}\label{lem:monotone}
\redcomment{Let $G$ be series-parallel, let $0<\mu<\lambda\le c(G)$, and let $H\subseteq G$ be an $s$-$t$ subgraph with $c(H)\ge\mu$. Then $\tau(G,c,\mu)\le\tau(G,c,\lambda)$ and $\tau(G,c,\mu)\le\tau(H,c,\mu)$.}
\end{lemma}

\begin{proof} Let $\xi$ be an SPF decomposition of $(G,s,t,c,\tau,\mu)$ with $\mathrm{dim}(\xi)=m$. If $m=1$, then $\tau(G,c,\mu)=\tau(G)\le\tau(G,c,\lambda)$. If $m\ge2$, then by Algorithm \ref{alg:flow} and Lemma~\ref{lem:path-cut} (the augmentation property), there is an SPF decomposition $\zeta$ of $(G,\lambda)$ with $\mathrm{dim}(\zeta)\ge m$ such that $\zeta^{(i)}=\xi^{(i)}$ for $i=1,\ldots,m-1$ and $\zeta^{(m)}$ is an augmentation of  $\xi^{(m)}$ with $\tau(\xi^{(m)})=\tau(\zeta^{(m)})$. By the non-decreasing property of the diameter sequence (recalling Algorithm \ref{alg:flow}), we have $\tau(\zeta^{(m)})\le\tau(G,c,\lambda)$, yielding $ \tau(G,c,\mu)=\tau(\xi^{(m)})\le\tau(G,c,\lambda)$.

Next, we prove $\tau(G,c,\mu)\le \tau(H,c,\mu)$ by graphical induction on $G$. The base case (where $G$ consists of a single edge) is obvious. Suppose that $G$ is a series- or parallel-connection of two smaller series-parallel graphs $G_1$ and $G_2$, for which the inequality holds

In the case of series-connection, the $s$-$t$ subgraph $H$ is a series-connection of two series-parallel graphs $H\cap G_i$, $i=1,2$. By the induction hypothesis $\tau(G_i,c,\mu)\le \tau(H\cap G_i,c,\mu)$, $i=1,2$, we obtain $\tau(G,c,\mu)=\tau(G_1,c,\mu)+\tau(G_2,c,\mu)\le\tau(H\cap G_1,c,\mu)+\tau(H\cap G_2,c,\mu)=\tau(H,c,\mu)$. In the case of parallel-connection, we observe that $H_i=H\cap G_i$ ($i\in{1,2}$) is either an $s$-$t$ series-parallel graph, or an edgeless graph on $\{s,t\}$. In the latter scenario, we specially define $\mu_i:=0$ \purplecomment{and} $\mu_{3-i}=\mu$. Moreover, for any $i=1,2$, both symbols $\tau(G_i,c,0)$ and $\tau(H_i,c,0)$ represent value $0$. In any case, $H$ is the edge-disjoint union of $H_1$ and $H_2$, and there exist nonnegative $\mu_1$ and $\mu_2$ such that $\mu_1+\mu_2=\mu$ and
\[
\tau(H,c,\mu)=\max\{\tau(H_1,c,\mu_1),\tau(H_2,c,\mu_2)\}\ge \max\{\tau(G_1,c,\mu_1),\tau(G_2,c,  \mu_2)\},
\] 
where  $\tau(G_i, c, \mu_i) \leq \tau(H_i, c, \mu_i)$, $i = 1, 2$, as guaranteed by the induction hypothesis.

Suppose without loss of generality that $\tau(G,c,\mu)=\tau(G_1,c,\lambda_1)\ge \tau(G_2,c,\lambda_2)$, where $\lambda_i$ is the flow value of $\xi$'s restriction to $G_i$ for $i=1,2$, and $\lambda_1+\lambda_2=\mu$. Assume that $\tau(G_2,c,\lambda_2+\epsilon)<\tau(G_1,c,\lambda_1)$ for some \redcomment{$0<\epsilon\le c(G_2)-\lambda_2$}. 
The restriction of $\xi$ to $G_2$ is an SPF decomposition of $(G_2,c,\lambda_2)$, denoted by $\eta$.  It follows that $G_2^{\eta,\mathrm{dim}(\eta)+1}$ has an $s$-$t$ path of length at most $\tau(G_2,c,\lambda_2+\epsilon)<\tau(G_1,c,\lambda_1)=\tau(G,c,\mu)$. When $\lambda_2=0$, for the empty decomposition $\eta$, symbol $G_2^{\eta,\mathrm{dim}(\eta)+1}$ stands for $G_2$. This path is also contained in $G^{\xi,m+1}$, whose restriction to $G_2$ is $G_2^{\eta,\mathrm{dim}(\eta)+1}$. It follows that $\tau(G^{\xi,m+1})\le\tau(G_2,c,\lambda_2+\epsilon)<\tau(G,c,\mu)$, giving a contradiction to  $\tau(G,c,\mu)=\tau(G^{\xi,m})\le\tau(G^{\xi,m+1})$. Thus
\begin{center}
    $\tau(G_2,c,\mu'_2)\ge\tau(G_1,c,\lambda_1)$ for all $\mu'_2\in(\lambda_2,c(G_2)]$.
\end{center} 
If $\mu_1\ge \lambda_1$, then $\tau(H,c,\mu)\ge \tau(G_1,c,\mu_1)\ge \tau(G_1,c,\lambda_1)=\tau(G,c,\mu)$, where the second inequality is implied the first part of the lemma. If $\mu_1< \lambda_1$, then 
$\mu_2>\lambda_2$ and $\tau(H,c,\mu)\ge \tau(G_2,c,\mu_2)\ge \tau(G_1,c,\lambda_1)=\tau(G,c,\mu)$, establishing the second part of the lemma.
\qed
    \end{proof}

\redcomment{The preceding results concern numerical invariants on general series-parallel graphs. They imply that the terminal SPF diameter cannot decrease when more flow is routed or edges are deleted; they do not identify path allocations, supports, or bottleneck cuts. For the rest of this subsection, $G$ is a BPA. Its unique-suffix property makes the upstream branches laminar, allowing us to uncross cuts and obtain a canonical rightmost cut.}
Given an instance $(G,s,t,{c}, \tau,\mu)$ in which $G=(V,E)$ be a BPA, let $F_1$ and $F_2$ be two minimum $s$-$t$ cuts in $G$. By Lemma~\ref{lem:path-cut}, all edges in $F_1$ are parallel, so are all edges in $F_2$. For any distinct edges $e_1,e_2\in F_1\cup F_2$ such that $e_1\prec e_2$, we have $|\{e_1,e_2\}\cap F_i|=1$ for $i=1,2$.

\begin{lemma}\label{lem:rightmost}
    If $G$ is a BPA, then $F:=(F_1\cap F_2)\cup\{e\in F_1\triangle F_2\mid \text{there exists }e'\in F_1\triangle F_2\text{ such that }e'\prec e\}$ is a minimum $s$-$t$ cut of $(G,c)$.
\end{lemma}
\begin{proof}
    We first prove that $F$ is an $s$-$t$ cut of $G$. If not, there exists an $s$-$t$ path $P$ such that $P\cap F=\emptyset$. Since $F_1, F_2$ are two $s$-$t$ cuts, there must be $\{e_1\}=P\cap F_1$ and $\{e_2\}=P\cap F_2$. If $e_1=e_2$, we have $e_1\in F_1\cap F_2\subseteq F$, thus $e_1\in P\cap F$, yielding a contradiction. If $e_1\neq e_2$, by Lemma\ref{lem:path-cut}, we will have $e_1\in F_1\setminus F_2$ and $e_2\in F_2\setminus F_1$, thus $e_1, e_2\in F_1\triangle F_2 $. Without loss of generality, suppose that $e_1\prec e_2$, thus we have $e_2\in P\cap F$, yielding a contradiction. Therefore, $F$ is a $s$-$t$ cut of $G$.

   We next prove that $F$ is a minimum cut. By the definition of $F$, for each $e\in F$ either $e\in F_1$ or $e\in F_2\setminus F_1$ and there exists $e_1\in F_1$ such that $e_1\prec e$. 
    \begin{align*}
        c(F)&=\sum_{e\in F\cap F_1}c_e+\sum_{e\in F\cap (F_2\setminus F_1)}c_e \\ &\le \sum_{e\in F\cap F_1}c_e+\sum_{e\in F_1\setminus F}c_e=c(F_1)\\
    \end{align*}
    \redcomment{To justify the inequality, for each $e\in F\cap(F_2\setminus F_1)$ let $A_e$ consist of the edges of $F_1$ preceding $e$. BPA structure makes these sets disjoint, with union $F_1\setminus F$. Replacing $e$ in $F_2$ by $A_e$ still gives an $s$-$t$ cut: every path through $e$ meets $F_1$ before $e$, since a later meeting would, by the unique suffix, contradict the presence of an earlier $F_1$ edge. Minimality of the capacity of $F_2$ gives $c_e\le c(A_e)$. Summing gives the displayed inequality, and hence $F$ is a minimum cut.}
\qed
\end{proof}

By Lemma \ref{lem:rightmost}, there is a unique minimum $s$-$t$ cut in $(G,c)$, denoted as $\mathitsc{R}(G)$, such that for any edge $e$ in any minimum $s$-$t$ cut of $G$, either $e$ or an ancestor of $e$ belongs to $\mathitsc{R}(G)$. We call $\mathitsc{R}(G)$ the \emph{rightmost minimum $s$-$t$ cut} of $(G,c)$ or $G$ (when $c$ is clear from the context). Let $E_1$ and $E_2$ be two subsets of $E$. We say that $E_2$ is \emph{on the right} of $E_1$ if no edge of $E_2$ is a descendant of some edge of $E_1$.


\redcomment{For a flow $\eta$ that is maximum on its support (with the original capacities $c$), let $\mathitsc R(\eta)$ denote the rightmost minimum cut of that support. Equivalently, it consists of the $\eta$-saturated edges with no strictly downstream $\eta$-saturated edge. We use this notation only when the support contains a saturated cut.}

\redcomment{For the following lemma only, we use the \emph{saturated-final-slice convention}: if Algorithm~\ref{alg:flow} truncates its last slice, augment that slice to a maximum flow in its residual shortest-path graph. Earlier slices and the terminal diameter are unchanged; the completed total can exceed the requested amount $\mu$. We slightly abuse the notations and write
\[
 X_i:=\sum_{h=1}^i\xi^{(h)},\qquad Z_i:=\sum_{h=1}^i\zeta^{(h)}.
\]
These are cumulative static flows, not individual slices. Corollary~\ref{cor:uni} gives $|X_i|=|Z_i|$ under the same convention. The next lemma strengthens this numerical invariance to equality of canonical cuts, even though the flows and their supports may differ.}

\begin{lemma}\label{lem:cutr}
\redcomment{Let $\xi$ and $\zeta$ be two SPF decompositions of a BPA instance with common dimension $m$, using the saturated-final-slice convention above. Then
$\mathitsc R(X_i)=\mathitsc R(Z_i)$ for every $1\le i\le m$.}
\end{lemma}
\begin{proof}
\redcomment{We first isolate the path-and-cut fact that makes the comparison possible.}

\redcomment{\emph{Claim.} If $K$ is a BPA and $\eta$ is a maximum flow in $(K,c)$, then $\mathitsc R(\eta)=\mathitsc R(K)$.}

\redcomment{To prove the claim, let $F$ consist of the $\eta$-saturated edges having no strictly downstream saturated edge. Every $s$-$t$ path in $K$ contains a saturated edge, since otherwise $\eta$ could be augmented along it. The last saturated edge of such a path belongs to $F$: its unique suffix contains every downstream edge. Thus $F$ is a cut. Its edges are pairwise incomparable, so each $s$-$t$ path meets it exactly once, and flow conservation gives $c(F)=|\eta|$. Hence $F$ is a minimum cut. Every minimum cut of $K$ is saturated by $\eta$; each of its edges therefore belongs to $F$ or precedes an edge of $F$. Consequently $F=\mathitsc R(K)$. Saturated edges carry positive flow, and flow conservation together with the unique suffix keeps every edge after them in the support. The set $F$ is therefore also the rightmost saturated cut of the support. This proves the claim.}

\redcomment{Now fix $i$, and let $K_i$ be the union of all $s$-$t$ paths of free-flow length at most $\tau^{(i)}$. This network depends only on the common slicing, not on either flow decomposition. In a BPA, a path is fixed by its first edge and the unique suffix thereafter; taking this union creates no additional, longer paths. Both $X_i$ and $Z_i$ are feasible flows in $K_i$. After the completed $i$th slice, no path of length at most $\tau^{(i)}$ has positive residual capacity: earlier shorter paths have already been exhausted, and the current shortest-path graph receives a maximum residual flow. By the augmentation property, both prefix flows are maximum in $K_i$. Applying the claim twice yields
\[
\mathitsc R(X_i)=\mathitsc R(K_i)=\mathitsc R(Z_i).
\]
This argument compares cuts in a common prefix network and does not infer edge saturation from equality of total flow values alone.}
\qed
\end{proof}

Given any $s$-$t$ path $P$ in $G$ and a nonnegative real number $r\le c(P)$, let $H$ be the subgraph of $G$ induced by $E\setminus\{e\in P:r=c_e\}$ and capacity $d\in\mathbb R^{E(H)}_{++}$ be defined by $d_e:=c_e$ if $e\not\in P$ and $d_e=c_e-r$ otherwise.

\begin{lemma}\label{lem:BPAcut1}
If $G$ is a BPA and $H$ is an $s$-$t$ subgraph of $G$, then the rightmost minimum $s$-$t$ cut $\mathitsc{R}(G)$ of $(G,c)$ is on the right of the rightmost minimum $s$-$t$ cut $\mathitsc{R}(H)$ of $(H,d)$.
\end{lemma}
\begin{proof}
 
 Suppose, to the contrary, that the statement does not hold. Then there exist edges $e_1\in \mathitsc{R}(G)$ and $e_2\in \mathitsc{R}(H)$ such that $e_1\prec e_2$. Let $G_{\nu(e_2)}$ denote the subgraph of $G$ consisting all $s$-$\nu(e_2)$ paths in $G$\purplecomment{, where $\nu(e_2)$ denote $e_2$'s tail vertices.} Since $e_1$ belongs to the rightmost minimum cut of $G$, we have $c(G_{\nu(e_2)})<c_{e_2}$. Because $H$ is a subgraph of $G$, it follows that $c(H\cap G_{\nu(e_2)})\le c(G_{\nu(e_2)})<c_{e_2}$, which contradicts the fact that $e_2$ lies in a minimum cut of $H$.
\qed
\end{proof}

Let $H^{(i)}$, $i=1,2,\ldots,m$ be as constructed in Algorithm \ref{alg:flow}. The following result follows immediately from Lemmas \ref{lem:BPAcut1} and \ref{lem:cutr}. 

\begin{corollary}\label{lem:BPAcut}
If $G$ is a BPA, then for each $i=1,\ldots,m-1$, the rightmost minimum $s$-$t$ cut of $(\cup_{h=1}^{i+1} H^{(h)},c)$ is on the right of the rightmost minimum $s$-$t$ cut of $(\cup_{h=1}^{i} H^{(h)},c )$.
\end{corollary}

\section{Nash Flow Dynamics}\label{sec:nash}


In this section, we focus on Nash flow dynamics on chains of BPAs. 

\subsection{Flows in a Single Block}
As a first step, in this subsection, we consider the case where $G$ is a single BPA with inflow rate $\mu(\theta)$, where $\mu(\theta)\leq \mu$ is a monotone non-decreasing \purplecomment{piecewise-constant}  function of time $\theta$ and satisfies $\mu(\theta)=0$  for $\theta<0$ . Let  $R_{\theta}$ denote the rightmost minimum cut of the dynamic shortest path network $G_\theta=(V,E'_{\theta})$, and let $c(G_\theta)$ denote the size of a minimum cut of the shortest path network $G_\theta$. 
Since $G$ is a BPA, its $s$-$t$ subgraph $G_\theta$ is also a BPA.

Recall that $\mathbf{x'}(\theta) = (x'_e(\theta))_{e \in E'_\theta}$ is a static $s$-$t$ flow on $G_\theta$ of value $\mu(\theta)$, see (\ref{char:static}) and Definition 2 in \cite{Cominetti2015}. As shown by the following lemma, if the inflow rate of time $\theta$ is no smaller than $c(G_\theta)$, and the particle with network entry time $\theta$ never queues after passing through  $R_{\theta}$, then, in the static flow,  the edges in $R_\theta$ share the total flow value $\mu(\theta)$ based on their capacity's proportion of the total capacity of the cut.

\purplecomment{
Before describing the properties of Nash flows, we first present several basic properties of flows over time. Consider a flow over time instance $(G,{c}, \tau,s,t,\mu)$, where $G$ is a BPA graph. For any vertex $v\in V$, let $G_v$ denote the subgraph of $G$ consisting of all paths in $\mathcal{P}_{s,v}$.
}
\purplecomment{
\begin{lemma}\label{lem:boundness}
    At any time $\theta$, for any vertex $v\in V(G)\setminus{s}$ and any edge $e\in \delta^+(v)$, $f^+_{e}(\theta)\le {c}(G_v)$.
\end{lemma}
}
\purplecomment{
\begin{proof}
    Since $G$ is BPA, we can assign a topological ordering to the vertices such that for every directed edge $(v,w)$, the label of $v$ is strictly less than that of $w$. Suppose $|V|=n$, we proceed by induction on $n$. For any vertex $u\in V(G)$, denote $ N^+(u)=\{v|e=uv\in E(G)\}$. When $n=2$, $v\in N^+(s)$, then for $e=vw$ $f^+_{e}(\theta)\le \sum _{e\in \delta^-(v)}{c}_e={c}(G_v)$. Suppose the lemma is correct on $i\le k$, then when $i=k+1$, 
    \begin{equation*}
        \begin{aligned}
            f^+_{e}(\theta)&\le \sum _{e\in \delta^-(v)}f^-_{e}(\theta)=\sum _{e\in \delta^-(v)}\min\{{c} _e, f^+_{e}(\theta)\} \\ &\le \sum _{e=uv\in \delta^-(v)}\min\{{c} _e, {c}(G_u)\}= {c}(G_v)
        \end{aligned}
    \end{equation*}
    The first equality holds due to the definition of $f^-_e(\theta)$. The second inequality holds due to the induction. The second equality can be easily proofed by induction on series-parallel graph. Therefore, the lemma is correct for any vertex $v\in V(G)\setminus{s}$ and any edge $e\in \delta^+(v)$, $f^+_{e}(\theta)\le {c}(G_v)$.  
\end{proof}
}

\purplecomment{
\begin{lemma}\label{lem:boundness2}
At any time $\theta$, for any edge $e$ after a minimum cut $R_{\theta}$ of $E'_{\theta}$, there is $f^+_{e}(l_v(\theta))< {c}_e$ .
\end{lemma}
}
\purplecomment{
\begin{proof}
For any edges $e=vw \in E'_{\theta} $, such that there exists edge(s) that is descendant of $e$, denote those edges as $C^e_{\theta}$. We will have 
\[f^+_{e}(l_v(\theta))\le  \sum_{e_0\in C^{e}_{\theta}}{c}_{e_0}< {c}_e.\] 
The first inequality follows by applying Lemma~\ref{lem:boundness} to the case $G=G[E'_\theta]$. The second inequality holds because if $ \sum_{e_0\in C^{e}_{\theta}}{c}_{e_0} \ge {c}_e$, the capacity of $R_{\theta}\setminus C^{e}_{\theta} \cup \{e\}$ is at least $R_{\theta}$, contradicting to the fact that $R_{\theta}$ is the rightmost minimum cut. 
\end{proof}
}

\begin{lemma}\label{lem:assign}
  For almost all $\theta$, if $c(G_\theta)\leq \mu(\theta)$ and no edge of $R_{\theta}$ is a descendant of any edge of $E_\theta^*$, then \purplecomment{for any $e$ in $R_{\theta}$, there is} $x'_e(\theta)=\frac{c_e}{c(G_\theta)}\cdot\mu(\theta)$. Moreover,  $l'_t(\theta)=\frac{\mu(\theta)}{c(G_\theta)}\ge 1$
\end{lemma}

\begin{proof}

\purplecomment{Note that the equilibrium characterization is right-continuous, there exists sufficiently small constant $\delta>0$ such that $G_\theta$ and $R_\theta$ remain unchanged during the interval $[\theta,\theta+\delta)$.  During the time interval $[\theta,\theta+\delta)$, for every edge $e$ that is an ancestor of some edges of $R_\theta$, by Lemma \ref{lem:boundness2}, $f_{e}^{+}(l_{\nu(e)}(\theta))<c_e$, i.e., no queue forms during the time interval $\big[l_{\nu(e)}(\theta), l_{\nu(e)}(\theta+\delta)\big)$. We first assume that for each edge $e$ in $R_\theta$ or is an ancestor of some edges of $R_\theta$, there is $x'_e(\theta)>0$. Then recursively using 
\[
\frac{x'_e(\theta)}{c_e} = \frac{f_e^+(l_{\nu(e)}(\theta))}{c_e} \cdot l'_{\nu(e)}(\theta)<l'_{\nu(e)}(\theta)
\]
and Theorem \ref{thm:eq}, we can prove that for every edge $e\in R_\theta$, it holds $l_t'(\theta)=l_{\omega(e)}'(\theta)\ge \frac{x'_e(\theta)}{c_e}$.   }

\purplecomment{By contradiction, suppose that the lemma fails.  Then there exist edges $a,b\in R_\theta$ such that 
\[
x'_a(\theta)\ge\frac{c_a}{c(R_\theta)}\cdot\mu(\theta)+\Delta,
\]
\[
x'_b(\theta)\le\frac{c_b}{c(R_\theta)}\cdot\mu(\theta)-\Phi,
\]
where $\Delta$ and $\Phi$ are positive constants.  Then from the above we know 
\[
\frac{x'_b(\theta)}{c_b}<\frac{\mu(\theta)}{c(R_\theta)}<\frac{x'_a(\theta)}{c_a}\le l_{\omega(a)}'(\theta)=l_t'(\theta)=l_{\omega(b)}'(\theta).
\]
By Theorem \ref{thm:eq} and the definition of $x_e'(\theta)$, it must be the case that $l_{\omega(b)}'(\theta)=l_{\nu(b)}'(\theta)$ and $f_b^+(l_{\nu(b)})(\theta)<c_b$.  }

\purplecomment{If there exists an  $s$-$\nu(b)$ path $P$ of $E_\theta'$ such that $\forall e\in P$,   $e\notin E_\theta^*$ and $f_{e}^+(l_{\nu(e)}(\theta))< c_e$, then we can deduce that $l_t'(\theta)=l_{\nu(b)}'(\theta)=\cdots=l_{s}'(\theta)=1$, contradiction. Otherwise, by the BPA structure, the subgraph (denoted as $B$) consisting of all $s$-$\nu(b)$ paths in $G_\theta$ contains an $s$-$\nu(b)$ cut $F$ such that each edge $e\in F$ either contains a queue at time $\theta$ or $f_{e}^+(l_{\nu(e)}(\theta))\ge c_e$, and no ancestors of any edge in $F$ (before $\nu(b)$) satisfying this. Then the outflow rate satisfies $f_{e}^-(l_{\omega(e)}(\theta))= c_e, \forall e\in F$ and $c(F)=\sum_{e\in F}f_{e}^-(l_{\omega(e)}(\theta))=f_b^+(l_{\nu(b)})(\theta)$. Recall that $c_b\le c(F)$. Contradiction.}

\purplecomment{Moreover, if there exists an edge $e\in R_\theta$ with $s$-$\nu(e)$ path $P$ of $E_\theta'$ such that $\forall e_P\in P$,   $e_P\notin E_\theta^*$ and $f_{e_P}^+(l_{\nu(e_P)}(\theta))< c_{e_P}$, then we can deduce that $l_{\nu(e)}'(\theta)=\cdots=l_{s}'(\theta)=1< \frac{x'_e(\theta)}{c_e}$. Therefore $l'_{\omega(e)}=\frac{x'_e(\theta)}{c_e}$. Otherwise, by the BPA structure, the subgraph (denoted as $B$) consisting of all $s$-$\nu(e)$ paths in $G_\theta$ contains an $s$-$\nu(e)$ cut $F$ such that each edge $e'\in F$ either contains a queue at time $\theta$ or $f_{e'}^+(l_{\nu(e')}(\theta))\ge c_{e'}$, and no ancestors of any edge in $F$ (before $\nu(e)$) satisfying this. Then the outflow rate satisfies $f_{e'}^-(l_{\omega(e')}(\theta))= c_{e'}, \forall e'\in F$. Hence, \[
f_e^+(l_{\nu(e)})(\theta)=\sum_{e'\in F}f_{e}^-(l_{\omega(e')}(\theta))=c(F)\ge c_e.
\]
By the definition of $x_{e}(\theta)$, it follows $l_{\nu(e)}'(\theta)=\frac{x'_e(\theta)}{f_e^+(l_{\nu(e)})(\theta)}\le \frac{x'_e(\theta)}{c_e}$. By Theorem \ref{thm:eq}, we have $l'_{\omega(e)}(\theta)=\max\{l_{\nu(e)}'(\theta),\frac{x'_e(\theta)}{c_e}\}=\mu(\theta)/c(G_\theta)$ for each $e\in R_\theta$, giving the result.}

\purplecomment{If each edge $e$ in $R_\theta$ or is an ancestor of some edges of $R_\theta$, the assumption that $x'_e(\theta)>0$ does not always hold. There must exists edge(s) $e$ in $R_\theta$ such that $x'_e(\theta)=0$. Thus, by Theorem \ref{thm:eq}, $l'_t(\theta)=\cdots= l'_s(\theta)=1$. However, for all edges in $R_\theta$ or an ancestor of some edges of $R_\theta$ with $x'_e(\theta)>0$, denoted as $R^{>0}_\theta=\{e\in R_\theta:x'_e(\theta)>0\}$, follow a similar
argument that there must be 
\[
x'_e(\theta)=\frac{c_e\cdot \mu(\theta)}{\sum_{e\in R^{>0}_\theta}c_e}=\frac{\mu(\theta)}{\sum_{e\in R_\theta}c_e}
\]
and 
\[
l'_t(\theta)=l'_{\omega(e)}(\theta)=\frac{\mu(\theta)}{\sum_{e\in R_\theta:x'_e(\theta)>0}c_e}>1,
\]
leading to a contradiction.\qed} 
\end{proof}

The above lemma shows that if the inflow rate of time $\theta$ is no smaller than $c(G_\theta)$, and the particle with network entry time $\theta$ never queues after passing through  $R_{\theta}$, and the travel latency $l_t(\theta)-\theta$ is monotone non-decreasing w.r.t. time $\theta$.

Recalling Section~\ref{sub:nde}, Algorithm~\ref{alg:flow} outputs an SPF decomposition $\xi$ for the instance with a constant network inflow rate $\mu\le c(G)$, which provides a sequence of shortest-path subgraphs $H^{(i)}$, $i=1,2,\ldots$ of the current graph $G^{\xi,i}$. By Corollary~\ref{cor:uni}, the instance has a unique {capacity-diameter slicing} $(c^{(i)},\tau^{(i)})_{i=1}^m$ independent of the flow decompositions.

\begin{lemma}\label{lem:dynamic}
If $\tau^{(i)}\leq l_t(\theta)-\theta<\tau^{(i+1)}$, then $c(G_\theta)=\sum_{h=1}^ic^{(h)}$, for every $i=1,\ldots,m-1$.
\end{lemma}

\begin{proof}

Set $\theta_1:=0$, and for any integer $2\le i\le m$,  define 
$\theta_i=\inf \{\theta>\theta_{i-1} \,|\, l_t(\theta)-\theta=\tau^{(i)}\}$. For convenience, an $s$-$v$ path $P$ is called \emph{active} at time $\theta$ if the flow particle entering $s$ at time $\theta$ can reach $v$ along $P$ at time $l_v(\theta)$. The high-level idea of the proof is to establish the following three facts by induction on $i$. Whenever $\tau^{(i)}\leq l_t(\theta)-\theta<\tau^{(i+1)}$, the following statements hold:
 \begin{itemize}
     \item[(i)] $R_{\theta_i}$ is always the rightmost minimum cut of $G_\theta$;
     \item[(ii)] The rightmost cut $R_{\theta}$ shifts rightward as time $\theta$ goes on. The particle with network entry time $\theta$ never queues after passing through (any edge of) $R_{\theta}$;
     \item[(iii)]  $E_{\theta_{i+1}}^*\cap H^{(i+1)}=\emptyset$, which implies that $H^{(i+1)}\subseteq E'_{\theta_{i+1}}$.
 \end{itemize}
By definition, $l_t(\theta)-\theta=\tau^{(i+1)}$ when $\theta=\theta_{i+1}$. The conclusion  $c(G_\theta)=\sum_{h=1}^ic^{(h)}$ of the lemma follows from statement (iii) instantly.

\paragraph{Base case: $i=1$.}  To prove (i), recall that $H^{(1)}$ denote the shortest paths graph of $G$, we have $E'_{\theta_i}= H^{(1)}$ and $R_{\theta_i}=\mathitsc{R}(H^{(1)})$. We prove that for each edge $e\in H^{(1)}$, $e\in E'_{\theta}$. On contrary, suppose that there exists an edge $e=vw\in H^{(1)}$ and $e=vw\notin E'_{\theta}$. Hence, it follows that there is no queue on any $s$-$v$ path $P_{sv}$ in the $E'_{\theta_1}$; otherwise, a flow particle in such a queue would necessarily traverse $e$ to reach $t$. Thus, at time $\theta$, for an $s$-$w$ path $P_{sv}\cup e$ that goes through $v$, the travel time from $s$ to $w$ is exactly $\tau(P_{sv}\cup \{e\})$. On the other hand, $e=vw\notin E'_{\theta}$, there exists a path $P_{sw}$ with travel time strictly shorter than the travel time of $P_{sv}\cup\{e\}$. This yields a contradiction since, by definition, the shortest path $P_{sv}\cup\{e\}$ is the shortest $s$-$w$ path. Since $l_t(\theta)-\theta<\tau^{(i+1)}$ then any path $p$ with $\tau(p)\ge \tau^{(i+1)}$ is not contain in the current $E'_{\theta}$, which implies that $c(G_\theta)\le c(H^{(1)})$. Moreover, $R_{\theta_1}=\mathitsc{R}(H^{(1)})$ is the rightmost minimum cut in $E'_{\theta}$.

    By (i) and the BPA structure, it is obvious that statement (ii) holds. To prove (iii), for any $s$-$t$ path $P\subseteq H^{(i+1)}$ with $\tau(P)=\tau^{(i+1)}$, for any edge $e=uv\in P$, if there is a queue on it, there must be particles who leave from source $s$ at some time $\rho \in [\theta_i,\theta_{i+1})=[0,\theta_{i+1})$ and travel through $e$ with $f^+_{e}(l_u(\rho))>c_{e}$. All particles that reach $t$ must travel through an edge $e_{\rho}$ in $R_{\theta_i}$,  which, by (i), is a minimum cut of $E'_{\rho}$.
    \begin{itemize}
        \item If $e\prec e_{\rho}$ or $e= e_{\rho}$, the path starting from $e$ is unique, which means $e_{\rho}\in P\subseteq H^{(i+1)}$. However $e_{\rho} \notin H^{(i+1)}$, a contradiction.
        \item If $e_{\rho}\prec e$, let  $C^e_{\theta}$ consists of the edges that are descendants of $e$ in $R_\rho$. Then, $c_{e}>\sum _{e'\in C^e_{\theta}}c_{e'}$ ($e\in H^{(i+1)}$). Moreover,
        \[
        f^+_{e}(l_u(\rho))\le \sum _{e'\in C^e_{\theta}}c_{e'}.
        \]
        Thus, $f^+_{e}(l_u(\rho))< c_{e}$, yielding to a contradiction.
    \end{itemize}
    
    Therefore, there is no queue on $P$, and for the particle with entry time $\theta_{i+1}$, the travel time along $P$ is $\tau^{(i+1)}$. It follows that $H^{(i+1)}\subseteq E'_{\theta_{i+1}}$ and $\mathitsc{R}(H^{(i+1)})\subseteq E'_{\theta_{i+1}}$. 


    \paragraph{Inductive step: $i\ge 2$} Assume that statements (i) -- (iii) are valid for all $1, 2, \ldots,i-1$.
    To justify (i) for index $i$, Note that at time $\theta=\theta_{i}$, $l_t(\theta)-\theta=\tau^{(i)}$. Then any $s$-$t$ path $P$ with $\tau(P)>\tau^{(i)}$ is not contained in the current $E'_{\theta_{i}}$. Thus, $R_{\theta_{i}}\le \sum_{j=1}^{i} c^{(j)}$. By induction hypothesis on (i), for any $e\in \mathitsc{R}(\cup_{j=1}^{i-1} H^{(j)})$, we can prove that $e\in E'_{\theta_{i}}$. Moreover, from the induction hypothesis on (iii), we have $\mathitsc{R}( H^{(i)})\subseteq E'_{\theta_{i}}$. Thus, $R_{\theta_{i}}=\mathitsc{R}(\cup_{j=1}^{i} H^{(j)})$. From Corollary \ref{lem:BPAcut}, for an $\epsilon>0$, the rightmost cut moves rightward over the $(\theta_{i}-\epsilon, \theta_{i}]$.

    For any time $\theta\in [\theta_i, \theta_{i+1})$, from the definition of $\tau^{(i+1)}$, we know that any $s$-$t$ path $P$ with $\tau(P)\ge \tau^{(i+1)}$ is not active. Thus, $c(G_\theta)\le \sum_{h=1}^i c^{(h)}$. By inductive hypothesis, we know that when time goes from $\theta_1$ to $\theta_i$, the rightmost minimum $s$-$t$ cut shifts rightward. If (i) does not hold, without loss of generality, suppose that $\theta$ is the first time such that $R_\theta\subsetneq R_{\theta_{i}}$. Thus, there is no queue on edges after $R_{\theta_i}$, which is the rightmost minimum cut at $\theta-\epsilon$ ($\epsilon\ge 0$ and is infinitesimal). Thus, in $G_\theta$, there is no queue after $R_\theta$. Let $e= vw\in R_{\theta_i}\setminus E'_\theta$. Then there exists $s$-$t$ path $P\subseteq E'_{\theta-\epsilon}$ that contains $e$. Moreover, there is no queue on edges $e'\in P[w,t]$ (because these edges are ancestors of some edges in the minimum cut), and there is no queue on edges $P[s,w]$ (if such a queue exists, the flow particle must travel through $e$). Moreover, the travel time of a particle entering $G$ at time $\theta-\epsilon$ and reaches $t$ along $P$ is $l_t(\theta-\epsilon)-(\theta-\epsilon)$. Thus, the travel time along $P$ for a particle entering $G$ at time $\theta$ is $\tau(P)$, which is equal to or less that $\lim_{\epsilon\rightarrow 0}l_{t}(\theta-\epsilon)-(\theta-\epsilon)= l_t(\theta)-\theta$ (by the continuity of $l_v(\cdot)$). However, $e\in P\nsubseteq E'_{\theta}$, yielding a contradiction.

   Again, statement (ii) follows from (ii) instantly. The proof for (iii) follows almost identically to the counterpart in the base case.
\qed
\end{proof}

\begin{lemma}\label{IncreasingArrival}
If $c(G_\theta)<\mu(\theta)$ and $\tau^{(i)}\leq \ell_t(\theta)-\theta<\tau^{(i+1)}$, then at time $l_t(\theta)$, the arrival flow rate at sink $t$  is exactly  $\sum_{h=1}^ic^{(h)}$, that is, $\sum_{e\in \delta^-(t)}f_e^{-}(l_t(\theta))=c(G_\theta)=\sum_{h=1}^ic^{(h)}$. 
\end{lemma}

\begin{proof}
  We expand $G$ by adding a virtual vertex $t^*$ and a virtual edge $e^*=tt^*$ outgoing from $t$ with zero length and infinite capacity. The Nash flow is also expanded accordingly. It is obvious that $\sum_{e\in \delta^-(t)}f_e^{-}(l_t(\theta))= f^+_{e^*}(l_t(\theta))$. Then by the definition of $x'_e(\theta)$ and Lemma~\ref{lem:assign}, we have
     \[
     \sum_{e\in \delta^-(t)}f_e^{-}(l_t(\theta))=\frac{x'_{e^*}(\theta)}{l'_t(\theta)}=\frac{\mu(\theta)}{\mu(\theta)/c(G_\theta)}=c(G_\theta)=\sum_{h=1}^ic^{(h)},
     \]
     confirming the lemma.
     \qed
\end{proof}

\begin{lemma}\label{SteadyArrival}
Whenever  $c(G_\theta)\geq \mu(\theta)$, i.e., the capacity of the dynamic minimum cut exceeds the inflow rate, it holds that $l_t'(\theta)=1$ and  $\sum_{e\in \delta^-(t)}f_e^{-}(l_t(\theta))=\mu(\theta)$, i.e., the arrival flow rate at sink $t$ is equal to $\mu(\theta)$.
\end{lemma}

\begin{proof}

We first prove that if $R_{\theta}\ge \mu$, then $l'_t(\theta)\le 1$.

\purplecomment{Note that the equilibrium characterization is right-continuous, there exists sufficiently small constant $\delta>0$ such that $G_\theta$ and $R_\theta$ remain unchanged during the interval $[\theta,\theta+\delta)$. Follow the same proof in  Lemma \ref{lem:assign}, for every edge $e\in R_\theta$, it holds $l_t'(\theta)=l_{\omega(e)}'(\theta)$.   }

\purplecomment{If there exists edge(s) $b$ in $R_\theta$ or is an ancestor of some edges of $R_\theta$, there is $x'_b(\theta)=0$. We can get 
\[
l'_t(\theta)\le \cdots\le l'_{\omega(b)}(\theta)\le l'_{\nu(b)}(\theta)\le \cdots \le l'_s(\theta)=1,
\]
proving the claim. Assuming that for each edge $e$ in $R_\theta$ or is an ancestor of some edges of $R_\theta$, there is $x'_e(\theta)>0$.}

\purplecomment{Note that $c(G_\theta)=\sum_{e\in R_\theta}c_e\ge \mu(\theta)$ and $\sum_{e\in R_\theta}x'_e(\theta)$. So if $x'_a(\theta)\ge c_a+\Delta$ for some edge $a\in R_\theta$, then $x'_b(\theta)\le c_b-\phi$ for another edge $b\in R_\theta$, where $\Delta$ and $\Phi$ are positive constants. Then from the above we know $\frac{x'_b(\theta)}{c_b}<1<\frac{x'_a(\theta)}{c_a}\le l_{\omega(a)}'(\theta)=l_t'(\theta)=l_{\omega(b)}'(\theta)$.  By Theorem \ref{thm:eq} and the definition of $x_e'(\theta)$, it must be the case that $l_{\omega(b)}'(\theta)=l_{\nu(b)}'(\theta)$ and $f_b^+(l_{\nu(b)})(\theta)<c_b$.  }

\purplecomment{If there exists an  $s$-$\nu(b)$ path $P$ of $E_\theta'$ such that $\forall e\in P$,   $e\notin E_\theta^*$ and $f_{e}^+(l_{\nu(e)}(\theta))< c_e$, then we can deduce that $l_t'(\theta)=l_{\nu(b)}'(\theta)=\cdots=l_{s}'(\theta)=1$, contradiction. Otherwise, by the BPA structure, the subgraph (denoted as $B$) consisting of all $s$-$\nu(b)$ paths in $G_\theta$ contains an $s$-$\nu(b)$ cut $F$ such that each edge $e\in F$ either contains a queue at time $\theta$ or $f_{e}^+(l_{\nu(e)}(\theta))\ge c_e$, and no ancestors of any edge in $F$ (before $\nu(b)$) satisfying this. Then the outflow rate satisfies $f_{e}^-(l_{\omega(e)}(\theta))= c_e, \forall e\in F$ and $c(F)=\sum_{e\in F}f_{e}^-(l_{\omega(e)}(\theta))=f_b^+(l_{\nu(b)})(\theta)$. Recall that $c_b\le c(F)$. Then for all edges in $R(\theta)$, $x'(e)\le c_e$} 

\purplecomment{If there exists an  $s$-$\nu(e)$ path $P$ of $E_\theta'$ such that $\forall e\in P$,  $e\notin E_\theta^*$ and $f_{e}^+(l_{\nu(e)}(\theta))< c_e$, then we can deduce that $l_{\nu(e)}'(\theta)=\cdots=l_{s}'(\theta)=1$, $l'_{\omega(e)}(\theta)\le \max\{l_{\nu(e)}'(\theta), \frac{x'_e(\theta)}{c_e}\}\le 1$. Otherwise, by the BPA structure, the subgraph (denoted as $B$) consisting of all $s$-$\nu(e)$ paths in $G_\theta$ contains an $s$-$\nu(e)$ cut $F$ such that each edge $e'\in F$ either contains a queue at time $\theta$ or $f_{e'}^+(l_{\nu(e')}(\theta))\ge c_{e'}$, and no ancestors of any edge in $F$ (before $\nu(e)$) satisfying this. Then the outflow rate satisfies $f_{e'}^-(l_{\omega(e')}(\theta))= c_{e'}, \forall e'\in F$. Hence, and $f_e^+(l_{\nu(e)})(\theta)=\sum_{e'\in F}f_{e}^-(l_{\omega(e')}(\theta))=c(F)\ge c_e$. By the definition of $x_{e}(\theta)$, it follows $l_{\nu(e)}'(\theta)=\frac{x'_e(\theta)}{f_e^+(l_{\nu(e)})(\theta)}\le \frac{x'_e(\theta)}{c_e}$. Hence, we have $l'_{\omega(e)}(\theta)=\frac{x'_e(\theta)}{c_e}\le 1$.
Alongside the fact that $l'_t(\theta)=\min_{e\in R_\theta}l'_{\omega(e)}(\theta)$, giving the result $l'_t(\theta)\le 1$. }

It remains to show that $l_t'(\theta)=1$.
Let $\theta_m$ be defined as in the proof of Lemma~\ref{lem:dynamic}. It follows that
\begin{align}\label{eq:ac}
         P_m\subseteq E'_{\theta_m}\text{ for any } s\text{-}t\; \text{path}\;P_m\subseteq H^{(m)} \; \text{with} \;\tau(P_m)=\tau^{(m)}.
    \end{align}
    Moreover, as $P_m\cap E^*_{\theta_m}=\emptyset$, we have $l'_t(\theta_m)\ge l'_s(\theta_m)=1$, leading to $l'_t(\theta_m)= 1$. 

    On the contrary, suppose that \purplecomment{$l'_t(\vartheta) \neq 1$} for some $\vartheta\ge \theta_m$ or \eqref{eq:ac} does not always hold. Set 
    \[
    \theta:=\inf\{\vartheta\ge \theta_m\,|\,l'_t(\vartheta)\neq 1\; \text{or}\; \eqref{eq:ac}\; \text{does not holds} \}.
    \]
    The flow particle that entering $s$ at $\theta$ reaches $t$ at $\theta+\tau^{(m)}$, Thus, any path $P$ with $\tau(P)>\tau^{(m)}$ is not contained in $E'_{\theta}$, implying that $R_{\theta}\le \sum_{i=1}^m c^{(i)}$. For any $e=vw\in \mathitsc{R}(H^{(m)})$, if $e$ is not contained in any cut in $E'_{\theta}$, let $\rho\le \theta$ be the earliest time at which $e$ is excluded from $E'_{\rho}$. For any $\epsilon>0$, then there exists $s$-$t$ path $P\subseteq E'_{\rho-\epsilon}$ containing $e$. There is no queue on edges in $ P[w,t]$, and there is no queue on edges $P[s,w]$. Moreover, the travel time of a particle entering $s$ at time $\rho-\epsilon$  along $P$ is $l_t(\rho-\epsilon)-(\rho-\epsilon)$. Thus, the travel time along $P$ of a particle with entry time $\rho$ is $\tau(P)$, which is equal to or less than $\lim_{\epsilon\rightarrow 0}l_{t}(\rho-\epsilon)-(\rho-\epsilon)= l_t(\rho)-\rho$ (by the continuity of $l_v(\cdot)$). However, $e\in P\nsubseteq
    E'_{\rho}$, yielding a contradiction. Thus, we have $c(G_{\theta})\ge \mu(\theta)$, implying $l'_t(\theta)\le 1$. 

    If $P_m\nsubseteq E'_{\theta}$, again $c(G_{\theta})<\mu(\theta)$. By Lemma \ref{lem:assign}, we obtain $l'_t(\theta)>1$. Hence, \eqref{eq:ac} holds at $\theta$.

    If $l'_t(\theta)< 1$ and for any $\eta<\theta$ there is $l'_t(\eta)=1$, we have $l'_t(\theta)-\theta=\tau^{(m)}$. Combine with the fact that $\tau(P_m)=\tau^{(m)}$ and $P_m\subseteq E'_{\theta}$, there is no queue on $P_m$, leading to $l'_t(\theta)\ge 1$, leading to a contradiction. Hence, such $\theta$ does not exist. 

    As a conclusion, we have if $c(G_\theta)\geq \mu(\theta)$, it holds that $l'_t(\theta)=1$, following a similar argument in Lemma \ref{IncreasingArrival}, there is $\sum_{e\in \delta^-(t)}f_e^{-}(l_t(\theta))=\mu(\theta)$.
    \qed
\end{proof}

\begin{corollary}\label{steady}
Once $c(G_\theta)\ge \mu \ge\mu(\theta)$, which implies $\mu(\theta)>\sum_{i=1}^{m-1}c^{(i)}$, it holds $l_t(\theta)-\theta=\tau^{(m)}$ a.e.. Specifically, if $c(G_\theta)\geq \mu =\mu(\theta)$, then the Nash flow reaches a steady state, and the corresponding arrival flow rate at sink $t$ will always be $\mu$, i.e., $\sum_{e\in \delta^-(t)}f_e^{-}(l_t(\theta))=\mu$.
\end{corollary}

Combining Lemma \ref{lem:assign} and Lemmas \ref{lem:dynamic}, \ref{IncreasingArrival}, and \ref{SteadyArrival} with Corollary~\ref{steady}, we have the following characterization of the Nash flow on BPA networks.
\begin{proposition}\label{char:BPAflow}
   Given any instance $(G,s,t,{c}, \tau,\mu(\theta))$, where $G$ is a BPA and $\mu(\theta)$ is a monotone non-decreasing \purplecomment{piecewise-constant} function of $\theta$ with maximum rate $\mu$. For its Nash flow, the travel latency $l_t(\theta)-\theta$ from source $s$ to sink $t$ is monotone non-decreasing with time $\theta$ and its maximum latency is $\tau^{(m)}$. Besides, the arrival time and arrival flow at the sink $t$ satisfy: 
   \begin{align*}
      l_t'(\theta)-1=
        \begin{cases}
            \max\{\frac{\mu(\theta)}{c^{(1)}},1\}-1 &\text{when}\quad l_t(\theta)-\theta\in [\tau^{(1)}, \tau^{(2)}),\\
            \;\vdots\\
            \max\{\frac{\mu(\theta)}{\sum_{i=1}^{m-1}c^{(i)}},1\}-1 &\text{when}\quad l_t(\theta)-\theta\in [\tau^{(m-1)}, \tau^{(m)}),\\
            0&\text{when}\quad l_t(\theta)-\theta= \tau^{(m)}.
        \end{cases}
    \end{align*}
    \begin{align*}
       \sum_{e\in \delta^-(t)}f_e^{-}(l_t(\theta))=
        \begin{cases}
            \min\{c^{(1)},\mu(\theta)\} &\text{when}\quad l_t(\theta)-\theta\in [\tau^{(1)}, \tau^{(2)}),\\
            \;\vdots\\
            \min\{\sum_{i=1}^{m-1} c^{(i)},\mu(\theta)\} &\text{when}\quad l_t(\theta)-\theta\in [\tau^{(m-1)}, \tau^{(m)}),\\
            \mu&\text{when}\quad l_t(\theta)-\theta= \tau^{(m)}.
        \end{cases}
    \end{align*}
\end{proposition}

 From Proposition \ref{char:BPAflow}, we can see that, if $G$ is a single BPA, and its inflow rate $\mu(\theta)$ is a monotone non-decreasing \purplecomment{piecewise-constant} function of $\theta$ with maximum rate $\mu$, then its social cost is $\tau^{(m)}$.

 \subsection{Flows in Chains of BPAs}
 We now turn to the general case where $G$ is a chain of BPAs, and show that the same conclusion continues to hold.
\begin{theorem}\label{thm:noBP}
 Given any instance $(G,s,t,c, \tau,\mu(\theta))$, where $G$ is a chain of BPAs and $\mu(\theta)$ is non-decreasing \purplecomment{piecewise constant} with $\max_\theta \mu(\theta)=\mu$. The social cost of its Nash flow equals $\tau(G, c, \mu)$. Consequently, it does not admit Braess's Paradox.
\end{theorem}
\begin{proof}
    Suppose that $G$ consists of a sequence of BPAs connected in series, denoted by $B_1, \ldots, B_k$, where the sink of $B_i$ coincides with the source of $B_{i+1}$ for $i=1,\ldots,k-1$. Since the inflow rate of $B_1$ is exactly $\mu(\theta)$, the above lemmas imply that the arrival flow at the sink of $B_1$, and hence the inflow rate of $B_2$, is monotone non-decreasing \purplecomment{piecewise constant} over time and finally reaches $\mu$. Repeatedly applying the same argument, we conclude that the inflow rate of each $B_i$ is monotone non-decreasing \purplecomment{piecewise constant} with maximum rate $\mu$.  Recall that $\tau(B_i, c, \mu)$ denotes the maximum path length $\tau^{(m)}(B_i)$ output by the Algorithm~\ref{alg:flow}'s decomposition on $B_i$. Then, by Proposition \ref{char:BPAflow},  the Nash flow will finally reach a steady state, and the maximum latency experienced by any flow particle traversing $B_i$ is exactly $\tau(B_i, c, \mu)$. Note that for $G$, its $\tau(G, c, \mu)=\sum_{i=1}^k \tau(B_i, c, \mu)$. So the maximum latency experienced by any particle traversing the network $G$ is exactly $\tau(G, c, \mu)$. 
    
   For any $s$-$t$ subgraph $H$ of $G$ with $c(H)\ge \mu$, apparently, $H$ is also a chain of BPAs. Hence, for the instance $(H,s,t,c, \tau,\mu(\theta))$, its maximum latency cost of the Nash flow equals $\tau(H, c, \mu)$. Recall Lemma \ref{lem:monotone} that $\tau(G, c, \mu)\le\tau(H, c, \mu)$. Therefore, Braess’s paradox cannot occur. This completes the proof.
   \qed
\end{proof}

Together, Theorems~\ref{thm:noBP} and \ref{thm:sufficient}, along with the topological description provided in Lemma~\ref{lem:topology}, resolve Conjecture~\ref{conj:BP}, yielding a necessary and sufficient condition for Braess’s paradox in flow over time.

\begin{theorem}
 \purplecomment{For a non-decreasing piecewise-constant inflow,} a network does not admit Braess’s paradox for flow over time if and only if it is a chain of BPAs.
\end{theorem}

\section{Monotonicity Conjecture}\label{sec:mono}
\redcomment{Here we keep total flow volume fixed and compare completion times, rather than the maximum latency under an indefinitely sustained inflow. The original conjecture below concerns constant inflow rates. Our extension uses nondecreasing, \purplecomment{piecewise-constant} rates on the injection intervals and the right-continuity convention from Section~\ref{sub:nashbp}.} \citet{Correa2021} propose a conjecture about the Nash flow, called \emph{monotonicity conjecture}, which basically states that, given an instance of the problem, the time it takes for an amount of flow to reach the sink $t$ is a decreasing function of the inflow rate $\mu$. 
They prove that if the monotonicity conjecture holds, the Price of Anarchy (PoA for short) for the flows over time,  measured as \redcomment{the worst-case ratio of the equilibrium makespan to the minimum time needed to route the same volume,} 
satisfies $\mathrm{PoA} \le \frac{e}{e-1}$.

\begin{conjecture}[Monotonicity Conjecture]
Consider a network $G$ and two fixed inflow rates $\mu_1 > \mu_2$, with corresponding dynamic equilibria. Let $T_1^{\mathrm{EQ}}$ and $T_2^{\mathrm{EQ}}$ denote the makespans for routing $M$ units of flow under these equilibria. Then $T_1^{\mathrm{EQ}} \le T_2^{\mathrm{EQ}}$.
\end{conjecture}

\citet{Correa2021} prove that this monotonicity conjecture holds for \emph{series concatenation of parallel-path networks}.  Note that the series concatenations of parallel-path networks are exactly the chains of parallel paths defined in Section \ref{conjecture}.

In this section, we extend the validity of the monotonicity conjecture to a broader class of network topologies and to non-uniform inflows. Unlike the proofs in \citep{Correa2021}, which rely primarily on analyzing the arrival time $l_t(M/\mu)$ under a uniform inflow rate $\mu$ at the source, we establish the monotonicity property through a more intuitive argument. In particular, we show that, for a fixed network instance, dominance of the inflow rate function implies dominance of the arrival-flow rate function. In the following, when mentioning an inflow rate $\mu(\theta)$, we always mean an inflow of total amount $M$, namely $\int_0^\infty \mu(\theta)d\theta=M$.  

Let $g_1(\theta)$ and $g_2(\theta)$ be two flow-rate functions with supports  $[0,T_1]$ and  $[0,T_2]$, respectively, satisfying \[\int_0^{T_1}g_1(\theta)d\theta=M \text{ and } \int_0^{T_2}g_2(\theta)d\theta=M.\] If $g_1(\theta)\ge g_2(\theta)$  for all $\theta\in[0,T_1]$, (then apparently $T_1\le T_2$ and) we say flow $g_1(\theta)$ \emph{dominates} flow $g_2(\theta)$.

Analogous to the previous section, we begin by analyzing the case in which $G$ is a single BPA.  Let $(c^{(i)},\tau^{(i)})_{i=1}^m$ be the capacity-diameter slicing  output by Algorithm~\ref{alg:flow} with the input $(G,s,t,c,\tau,c(G))$. Note that now $\sum_{i=1}^m c^{(i)}=c(G)$.  Let $\mu(\theta)$ be a monotone non-decreasing \purplecomment{piecewise-constant} function within support $[0,T^*]$ and $\int_0^{T^*}\mu(\theta)d\theta=M$. From Proposition \ref{char:BPAflow} we know, the travel latency function $(l_t(\theta)-\theta)$ is monotone non-decreasing with $\theta\in[0,T^*]$. Suppose that $l_t(T^*)-T^*\in [\tau^{(k)},\tau^{(k+1)})$ when the inflow terminates\footnote{If $k=m$, then we denote $\tau^{(k+1)}:=\infty$.}. With a slight abuse of notation, for $i=1,\ldots,k$, let $T^{i}$ and $\hat T^{i}$ denote the earliest and latest times, respectively, at which $l_t(T^i)-T^i=\tau^{(i)}$, and $\nu(\theta):= \sum_{e\in \delta^-(t)}f_e^{-}(\theta)$ denote the sink $t$'s arrival flow rate at time $\theta$. Then by Proposition \ref{char:BPAflow}, sink $t$'s arrival flow rate is a monotone non-decreasing \purplecomment{piecewise-constant} function, and its flow rate $\nu(l_t(\theta))$ at time $l_t(\theta)$ satisfies 
\begin{align*}
      \nu(l_t(\theta))=
        \begin{cases}
         \mu(\theta) &\text{ for } \theta\in [T^{i},  \hat T^{i}), \ \text{ for }i=1,\ldots,k\\
         \sum_{j=1}^ic^{(j)}
            &\text{ for }  \theta\in [\hat T^{i}, T^{i+1}), \text{ for }i=1,\ldots,(k-1)\\ 
             \sum_{j=1}^{k} c^{(j)} &\text{ for } \theta\in[\hat T^k,T^*).
        \end{cases}
    \end{align*}
Note that, the time $\hat T^i$ denote the only time that $\mu(\hat T^i)>\sum_{j=1}^{i} c^{(j)}$, and $l_t(\hat T^i)-\hat T^i=\tau^{(i)}$. If for some $i$ that $T^i=\hat T^i$, then the interval$ [T^{i}, \hat T^{i})$ is empty.  
As $l_t(\theta)$ is a strictly monotone increasing function of $\theta$, we can rewrite the arrival flow rate $\nu(\theta)$ at time $\theta$ as the following monotone non-decreasing \purplecomment{piecewise-constant} function:
\begin{align}\label{rewrite}
      \nu(\theta)=
        \begin{cases}
           \ 0 &\text{ for } \theta\in [0,\tau^{(1)}),\\
            \ \mu\big(\theta-\tau^{(i)}\big) &\text{ for } \theta\in [T^i+\tau^{(i)},\hat T^i+\tau^{(i)}), \ \text{ for } i=1,\ldots,k \\
            \sum_{j=1}^{i} c^{(j)}, \ &\text{ for }  \theta\in [\hat T^{i}+\tau^{(i)}, T^{i+1}+\tau^{(i+1)}), \text{ for } i=1,\ldots,(k-1)\\
             \sum_{j=1}^{k} c^{(j)} &\text{ for } \theta\in[\hat T^k+\tau^{(k)},l_t(T^*)).
        \end{cases}
    \end{align}

\begin{lemma}\label{lem:dominance}
    Given any network instance $(G,s,t,c,\tau)$ in which $G$ is a BPA, let $\mu_1(\theta)$ and  $\mu_2(\theta)$ be two monotone non-decreasing \purplecomment{piecewise-constant} inflow rate functions with supports  $[0,T_1^*]$ and $[0,T_2^*]$, and  let $\nu_1(\theta)$ and $\nu_2(\theta)$  be the corresponding arrival flow rates at the sink $t$. If $\mu_1(\theta)$ dominates $\mu_2(\theta)$, then $\nu_1(\theta)$ also dominates $\nu_2(\theta)$. 
\end{lemma}

\begin{proof}
Since we are dealing with two different inflows at the same time, to avoid confusion, we use $T^i_{\mu_1}$, $\hat T^i_{\mu_1}$ for  $i=1,2,\cdots,k_{\mu_1}$ and $T^i_{\mu_2}$, $\hat T^i_{\mu_2}$ for  $i=1,2,\cdots,k_{\mu_2}$  to denote the above defined times with respect to inflow functions $\mu_1(\theta)$ and $\mu_2(\theta)$, respectively. Let \purplecomment{$T_{\mu_1}^{EQ}:=(l_t(T_1^*))_{\mu_1(\theta)}$ and $T_{\mu_2}^{EQ}:=(l_t(T_2^*))_{\mu_2(\theta)}$}  
denote the last particle's arrival time at the sink with inflows $\mu_1(\theta)$ and $\mu_2(\theta)$, respectively.  Since $\mu_1(\theta)$  and $\mu_2(\theta)$ both are non-decreasing \purplecomment{piecewise-constant} functions, from equation (\ref{rewrite}) of the arrival flow we know, $\nu_1(\theta)$ and $\nu_2(\theta)$ both are monotone non-decreasing \purplecomment{piecewise-constant} functions with supports  $[0,T_{\mu_1}^{EQ}]$ and  $[0,T_{\mu_2}^{EQ}]$ correspondingly, and satisfy $\int_0^{T_{\mu_1}^{EQ}}\nu_1(\theta)d\theta=M$,  $\int_0^{T_{\mu_2}^{EQ}}\nu_2(\theta)d\theta=M$.

Following the definition of the dominance, to prove $\nu_1(\theta)$ dominates $\nu_2(\theta)$, we only need to prove $\nu_1(\theta)\ge \nu_2(\theta)$ for all $\theta\in[0,T_{\mu_1}^{EQ}]$.   Since $\nu_1(\theta)$ and $\nu_2(\theta)$ both are monotone non-decreasing \purplecomment{piecewise-constant}, and $\mu_1(\theta-\tau^{(i)})\ge \mu_2(\theta-\tau^{(i)})$ for every $\theta$, it \purplecomment{is} enough to prove $T^i_{\mu_1}\le T^i_{\mu_2}$ and  $\hat T^i_{\mu_1}\le \hat T^i_{\mu_2}$ for $i=1,\ldots,k_{\mu_1}$. We will prove this by induction analysis.

First   $T^1_{\mu_1}= T^1_{\mu_2}=0$. 
Recall that $\hat T^1_{\mu_1}$ and $\hat T^1_{\mu_2}$ denotes the first time $\theta$ that $\mu_1(\theta)>c^{(1)}$ and $\mu_2(\theta)>c^{(1)}$, respectively. Since $\mu_1(\theta)\ge \mu_2(\theta)$, thus $\hat T^1_{\mu_1}\le\hat T^1_{\mu_2}$. Suppose it holds that $T^j_{\mu_1}\le T^j_{\mu_2}$ and  $\hat T^j_{\mu_1}\le \hat T^j_{\mu_2}$ for $j=1,\ldots,i$. We prove the correctness for $i+1$. 

Suppose to the contrary that $T^{i+1}_{\mu_1}> T^{i+1}_{\mu_2}$. By the definition \purplecomment{of $\hat T^i$} and Proposition \ref{char:BPAflow}, we know
\purplecomment{$$\tau^{(i+1)}-\tau^{(i)}=\int_{\hat T^i_{\mu_2}}^{T^{i+1}_{\mu_2}}(l_t'(\theta)-1)_{\mu_2}d\theta=\int_{\hat T^i_{\mu_2}}^{T^{i+1}_{\mu_2}}\left(\frac{\mu_2(\theta)}{\sum_{j=1}^{i}c^{(j)}}-1\right)\,d\theta.$$ }
As $[\hat T^i_{\mu_1},T^{i+1}_{\mu_1}]\supset[\hat T^i_{\mu_1},T^{i+1}_{\mu_2}]$ , we have  $$\int_{\hat T^i_{\mu_1}}^{T^{i+1}_{\mu_2}}(l_t'(\theta)-1)_{\mu_1}d\theta=\int_{\hat T^i_{\mu_1}}^{T^{i+1}_{\mu_2}}\left(\frac{\mu_1(\theta)}{\sum_{j=1}^{i}c^{(j)}}-1\right)\,d\theta\ge \int_{\hat T^i_{\mu_2}}^{T^{i+1}_{\mu_2}}\left(\frac{\mu_2(\theta)}{\sum_{j=1}^{i}c^{(j)}}-1\right)\,d\theta= \tau^{(i+1)}-\tau^{(i)}.$$
The inequality is by the facts that $\hat T^i_{\mu_1}\le \hat T^i_{\mu_2}$ and $\mu_1(\theta)\ge \mu_2(\theta)$. This means in the instance with inflow rate $\mu_1(\theta)$, the travel latency  would reach $\tau^{(i+1)}$ at an earlier time $T^{i+1}_{\mu_2}$, contradicting the definition of $T^{i+1}_{\mu_1}$. So $T^{i+1}_{\mu_1}\le T^{i+1}_{\mu_2}$. 

Still recall that $\hat T^{i+1}_{\mu_x}$ is the earliest time  $\theta\ge T^{i+1}_{\mu_x}$ such that $\mu_x(\theta)>\sum_{j=1}^{i+1} c^{(j)}$ for $x=1,2$.  Since $\mu_1(\theta)$ dominates $\mu_2(\theta)$, we have $\hat T^{i+1}_{\mu_1}\leq \hat T^{i+1}_{\mu_2}$. This completes the proof.
\qed
\end{proof}

Having the above dominance result for a single BPA, then we can extend it to chains of BPAs using a similar argument to the proof of Theorem \ref{thm:noBP}.
\begin{theorem}
     Given any network instance $(G,s,t,c,\tau)$ with $G$ being a chain of BPAs.  Let $\mu_1(\theta)$ and  $\mu_2(\theta)$ be two monotone non-decreasing \purplecomment{piecewise-constant} inflow rate functions at the source $s$ with supports  $[0,T_1]$ and  $[0,T_2]$, and  let $\nu_1(\theta)$ and $\nu_2(\theta)$  be the corresponding arrival flow rates at the sink $t$.   If $\mu_1(\theta)$ dominates $\mu_2(\theta)$, then $\nu_1(\theta)$ also dominates $\nu_2(\theta)$.
\end{theorem}
\begin{proof}
  Suppose that $G$ consists of a sequence of BPAs connected in series, denoted by $B_1, \ldots, B_m$. Let $s_i$ be the source of $B_i$. Then  for $i=1,\ldots,m-1$, the sink of $B_i$ is also the source of $B_{i+1}$. When we are given the inflow $\mu_1(\theta)$ at $s=s_1$,  for $i=1,\ldots,m$, we denote the arrival flow at $s_{i+1}$ by $\nu^i_1(\theta)$, which is also the inflow of the next block.  When we are given the inflow  $\mu_2(\theta)$ at $s=s_1$, for $i=1,\ldots,m$,  we denote the arrival flow at $s_{i+1}$ by $\nu^i_2(\theta)$. Since $B_1$ is a single BPA, according to the given conditions and by Lemma \ref{lem:dominance}, we can deduce that  $\nu_1^1(\theta)$ and $\nu_2^1(\theta)$ both are monotone non-decreasing \purplecomment{piecewise-constant} inflow rate functions and  $\nu_1^1(\theta)$ dominates $\nu_2^1(\theta)$. As each block  $B_i$ is also a single BPA, repeated application of the lemma yields that $\nu_1^i(\theta)$ dominates $\nu_2^i(\theta)$ for $i=2,\ldots,m$. Finally, noting that $\nu_1(\theta)=\nu_1^m(\theta)$ and $\nu_2(\theta)=\nu_2^m(\theta)$, we conclude that $\nu_1(\theta)$ dominates $\nu_2(\theta)$.
 \qed
\end{proof}
\begin{corollary}\label{thm:conj}
    The Monotonicity Conjecture holds for any instance $(G,s,t,c,\tau)$ in which $G$ is a chain of BPAs and the inflow rate function is monotone non-decreasing \purplecomment{piecewise constant}.
\end{corollary}

\section{Conclusion}\label{sec:con}

This work develops a structural perspective on flow over time by examining  Nash flow's dynamic performance under a specific class of network topologies. The introduction of chain of BPAs graph identifies a class of networks that exhibits several desirable properties with respect to Nash flow. In particular, the structure of \purplecomment{chains} of BPAs prevents the emergence of Braess’s paradox, \redcomment{ensuring that edge removal cannot improve the maximum equilibrium latency.}. This completely resolves \citet{Macko2013}'s conjecture on the sufficient and necessary condition for the occurrence of Braess’s paradox. At the same time, we establish a dominance-preservation result for chains of BPAs, \redcomment{confirming the monotonicity conjecture of \citet{Correa2021} for chains of BPAs and extending the comparison to nondecreasing, piecewise-constant inflows on their injection intervals, under the stated equilibrium regularity. The conjecture on general networks remains open.} In addition, some of our network decomposition results for series–parallel networks may be of independent interest.

Several directions for future work remain. One promising avenue is to investigate Braess’s paradox under alternative social-cost objectives and characterize the network structures that excludes the paradox. In addition, extending the monotonicity property to more general classes of graphs, such as series-parallel graphs, beyond the chain of BPAs graph remains an important open problem. 

%
%
%
%
%

\bibliography{sample-bibliography}

\appendix




\medskip

\end{document}